\documentclass[reqno,a4paper]{article}
\usepackage{etex}
\usepackage[T1]{fontenc}
\usepackage[utf8]{inputenc}
\usepackage{lmodern}

\usepackage[style=alphabetic,firstinits=true,backref,backend=biber,isbn=false,url=false,maxbibnames=99]{biblatex}

\renewbibmacro{in:}{}
\newbibmacro{string+doi}[1]{\iffieldundef{doi}{\iffieldundef{url}{#1}{\href{\thefield{url}}{#1}}}{\href{http://dx.doi.org/\thefield{doi}}{#1}}}
\DeclareFieldFormat*{title}{\usebibmacro{string+doi}{\emph{#1}}}
\DefineBibliographyStrings{english}{%
  backrefpage = {$\uparrow$~p{.}},
  backrefpages = {$\uparrow$~pp{.}},
}

\usepackage{amssymb,amsmath,amsthm}
\usepackage{units}
\usepackage{graphicx}
\usepackage[all]{xy}
\usepackage{tikz}
\usetikzlibrary{arrows,calc,positioning,decorations.pathreplacing}
\usepackage{relsize}
\usepackage{mhsetup}
\usepackage{mathtools}
\usepackage{stmaryrd}
\usepackage{tocloft}
\usepackage{paralist} 
\usepackage[left=3cm,right=3cm,top=3cm,bottom=3cm]{geometry} 
\usepackage[colorlinks,citecolor=blue,linkcolor=blue,urlcolor=blue,filecolor=blue,bookmarksopen,bookmarksnumbered,bookmarksdepth=4,bookmarksopenlevel=2,breaklinks=true]{hyperref}
\usepackage{footnotebackref}

\definecolor{lightblue}{rgb}{0.8,0.8,1}

\numberwithin{equation}{section}
\numberwithin{figure}{section}

\newtheoremstyle{italicised}
        {\topsep}{\topsep}  
        {\itshape}  
        {}  
        {\bfseries}  
        {}  
        {1ex}  
        {}  
\theoremstyle{italicised}
\newtheorem{thm}{Theorem}[section]
\newtheorem{lem}[thm]{Lemma}
\newtheorem{prop}[thm]{Proposition}
\newtheorem{coro}[thm]{Corollary}

\newtheorem{athm}{Theorem}

\newtheorem{acoro}[athm]{Corollary}

\newtheoremstyle{upright}
        {\topsep}{\topsep}  
        {\upshape}  
        {}  
        {\bfseries}  
        {}  
        {1ex}  
        {}  
\theoremstyle{upright}
\newtheorem{defn}[thm]{Definition}
\newtheorem{rmk}[thm]{Remark}
\newtheorem{eg}[thm]{Example}

\newtheorem{notation}[thm]{Notation}

\newtheorem{assumption}[thm]{Assumption}

\newtheoremstyle{italicised-restate}
        {\topsep}{\topsep}  
        {\itshape}  
        {}  
        {\bfseries}  
        {}  
        {1ex}  
        {\thmname{#1}\thmnote{ \bfseries #3}}  
\theoremstyle{italicised-restate}

\makeatletter
\renewcommand*{\@seccntformat}[1]{\upshape\csname the#1\endcsname.\hspace{1ex}}
\renewcommand*{\section}{\@startsection{section}{1}{\z@}%
  {2.5ex \@plus 1ex \@minus 0.2ex}%
  {1.5ex \@plus 0.2ex}%
  {\large\bfseries}}
\renewcommand*{\subsection}{\@startsection{subsection}{2}{\z@}%
  {2.5ex \@plus 1ex \@minus 0.2ex}%
  {-1.5ex \@plus -0.2ex}%
  {\normalfont\normalsize\bfseries}}
\renewcommand*{\subsubsection}{\@startsection{subsubsection}{3}{\z@}%
  {2.5ex \@plus 1ex \@minus 0.2ex}%
  {-1.5ex \@plus -0.2ex}%
  {\normalfont\normalsize\bfseries}}
\renewcommand*{\paragraph}{\@startsection{paragraph}{4}{\z@}%
  {2.5ex \@plus 1ex \@minus 0.2ex}%
  {-1.5ex \@plus -0.2ex}%
  {\normalfont\normalsize\bfseries}}
\renewcommand*{\subparagraph}{\@startsection{subparagraph}{5}{\z@}%
  {2.5ex \@plus 1ex \@minus 0.2ex}%
  {-1.5ex \@plus -0.2ex}%
  {\normalfont\normalsize\slshape}}
\makeatother

\newcommand{\cB}{\mathcal{B}}
\newcommand{\cC}{\mathcal{C}}

\newcommand{\cG}{\mathcal{G}}

\newcommand{\cI}{\mathcal{I}}

\newcommand{\cL}{\mathcal{L}}

\newcommand{\cP}{\mathcal{P}}

\newcommand{\bC}{\mathbb{C}}

\newcommand{\bF}{\mathbb{F}}

\newcommand{\bN}{\mathbb{N}}

\newcommand{\bQ}{\mathbb{Q}}
\newcommand{\bR}{\mathbb{R}}

\newcommand{\bZ}{\mathbb{Z}}

\newcommand{\Emb}{\mathrm{Emb}}
\newcommand{\hconn}{\ensuremath{h\mathrm{conn}}}
\newcommand{\image}{\ensuremath{\mathrm{im}}}
\newcommand{\degree}{\ensuremath{\mathrm{deg}}}
\newcommand{\kernel}{\mathrm{ker}}
\newcommand{\coker}{\mathrm{coker}}
\newcommand{\Aut}{\mathrm{Aut}}
\newcommand{\Diff}{\mathrm{Diff}}

\newcommand{\undn}{\ensuremath{\underline{n}}}
\newcommand{\mbar}{\ensuremath{{\,\,\overline{\!\! M\!}\,}}}
\newcommand{\nbar}{\ensuremath{{\,\,\overline{\!\! N}}}}

\renewcommand{\geq}{\geqslant}
\renewcommand{\leq}{\leqslant}
\renewcommand{\footnoterule}{%
  \kern -3pt
  \hrule width \textwidth height 0.4pt
  \kern 2.6pt
}

\newcommand{\incl}[3][right]%
{%
\draw[<-,>=#1 hook] #2 to ($ #2!0.5!#3 $);
\draw[->,>=stealth'] ($ #2!0.5!#3 $) to #3;%
}
\newcommand{\inclusion}[5][right]%
{%
\draw[<-,>=#1 hook] #4 to ($ #4!0.5!#5 $) node[#2,font=\small]{#3};
\draw[->,>=stealth'] ($ #4!0.5!#5 $) to #5;%
}

\newcommand{\emb}{\ensuremath{\hookrightarrow}}
\newcommand{\ab}{\ensuremath{\mathsf{Ab}}}

\newcommand{\scard}{\ensuremath{\left\lvert S\right\rvert}}
\newcommand{\rcard}{\ensuremath{\left\lvert R\right\rvert}}

\newcommand{\undnplusonescript}{\ensuremath{\underline{\raisebox{0pt}[0pt][0pt]{\ensuremath{\scriptstyle n\!+\!1}}}}}
\newcommand{\Se}{\ensuremath{\mathsf{Se}^{\mathsf{fin}}}}
\newcommand{\Setp}{\ensuremath{\mathsf{Set}_*^{\mathsf{fin}}}}

\newcommand{\cBf}{\ensuremath{\mathcal{B}_{\mathsf{f}}}}

\newcommand{\catr}{\ensuremath{\mathsf{Cat}_R}}
\newcommand{\catrr}{\ensuremath{\mathsf{Cat}^R}}
\newcommand{\topr}{\ensuremath{\mathsf{Top}_R}}

\newcommand{\toprr}{\ensuremath{\mathsf{Top}^R}}
\newcommand{\grrmod}{\ensuremath{\text{gr-}R\text{-mod}}}
\newcommand{\Ubeta}{\ensuremath{\boldsymbol{\mathcal{U}\beta}}}

\newcommand{\cf}{\textit{cf}.\ }
\newcommand{\Cf}{\textit{Cf}.\ }

\newenvironment{itemizeb}%
{\begin{compactitem}

}%
{\end{compactitem}}

\begin{document}
\title{\Large\bfseries Twisted homological stability for configuration spaces\vspace{-1ex}}
\author{\small Martin Palmer\quad $/\!\!/$\quad 19 December 2017\vspace{-1ex}}
\date{}
\maketitle
{
\makeatletter
\renewcommand*{\BHFN@OldMakefntext}{}
\makeatother
\footnotetext{\textit{Address}: Mathematisches Institut der Universit{\"a}t Bonn, Endenicher Allee 60, 53115 Bonn, Germany}
\footnotetext{\textit{Email address}: \textsf{palmer@math.uni-bonn.de}}
\footnotetext{[---Also available at \href{https://mdp.ac/papers/ths-for-cs}{mdp.ac/papers/ths-for-cs}, where any addenda or informal related notes will also be posted.---]}
}
\begin{abstract}
Let $M$ be an open, connected manifold. A classical theorem of McDuff and Segal states that the sequence $\{C_n(M)\}$ of configuration spaces of $n$ unordered, distinct points in $M$ is \emph{homologically stable} with coefficients in $\bZ$ -- in each degree, the integral homology is eventually independent of $n$. The purpose of this paper is to prove that this phenomenon also holds for homology with twisted coefficients. We first define an appropriate notion of \emph{finite-degree twisted coefficient system} for $\{C_n(M)\}$ and then use a spectral sequence argument to deduce the result from the untwisted homological stability result of McDuff and Segal. The result and the methods are generalisations of those of Betley~\cite{Betley2002Twistedhomologyof} for the symmetric groups.
\end{abstract}
\tableofcontents


\section{Introduction}\label{sIntro}

For a pair of spaces $M$ and $X$, the \emph{configuration space of $n$ unordered points in $M$ with labels in $X$} is defined by
\[
C_n(M,X) \coloneqq (\mathrm{Emb}(n,M) \times X^n)/\Sigma_n.
\]
Here $n$ is the discrete space of cardinality $n$, so $\Emb(n,M)$ is the subspace of $M^n$ where no two points coincide. The symmetric group $\Sigma_n$ acts diagonally, permuting the points and the list of labels, so an element of $C_n(M,X)$ is a subset of $M$ of cardinality $n$, together with an element of $X$ ``attached'' to each point.

\begin{assumption}\label{Assumption}
Henceforth we assume that $M$ is an open, connected manifold with $\dim(M)\geq 2$, and that $X$ is a path-connected space. To be precise, by an \emph{open} manifold we mean a manifold with empty boundary, each of whose (path-)components is non-compact but paracompact.
\end{assumption}

Since $M$ is open, there are well-defined ``stabilisation maps'' $C_n(M,X)\to C_{n+1}(M,X)$, which we define precisely in \S\ref{ssTCS-stabmap} below. They are so called because the sequence of spaces $\{C_n(M,X)\}$ is homologically stable with respect to them:

\begin{thm}[\cite{Segal1973Configurationspacesand, McDuff1975Configurationspacesof, Segal1979topologyofspaces, Randal-Williams2013Homologicalstabilityunordered}]\label{tUntwistedStability}
Under the conditions on $M$ and $X$ assumed above, the map $C_n(M,X)\to C_{n+1}(M,X)$ induces an isomorphism on integral homology in degrees $*\leq \frac{n}{2}$, and is split-injective on homology in all degrees.
\end{thm}

\subsection{Twisted homological stability.}

Several other families of groups or spaces which are homologically stable are also known to have homological stability for \emph{twisted coefficients}. For example general linear groups \cite{Dwyer1980Twistedhomologicalstability}, mapping class groups of surfaces \cite{Ivanov1993homologystabilityTeichmuller, CohenMadsen2009Surfacesinbackground, Boldsen2012Improvedhomologicalstability} and the symmetric groups \cite{Betley2002Twistedhomologyof} are known to satisfy this phenomenon. A machine for proving twisted homological stability for many natural families of groups is constructed in \cite{Randal-WilliamsWahl2017Homologicalstabilityautomorphism}, and in particular covers the cases of mapping class groups of non-orientable surfaces and orientable $3$-manifolds.

The minimum data required in order to pose the question of twisted homological stability for a sequence of based, path-connected spaces $\{Y_n\}$ is a functor $\pi_1(\{Y_n\}) \to \mathsf{Ab}$, where the source is the category (groupoid) where the objects are the natural numbers, all morphisms are automorphisms and $\mathrm{Aut}(n) = \pi_1(Y_n)$. In other words, this is just a choice of $\pi_1(Y_n)$-module for each $n$. There is of course no chance of stability with respect to such a general ``twisted coefficient system'', as the $\pi_1(Y_n)$-modules for differing $n$ may be completely unrelated.

To obtain a notion of \emph{twisted coefficient system} with a chance of stability, one needs to add some (non-endo)morphisms to $\pi_1(\{Y_n\})$ and require that the functor from this new source category to $\mathsf{Ab}$ satisfy some finiteness conditions defined in terms of the new morphisms. The correct way to do this depends on the particular context one is working in (although a very general context for classifying spaces of discrete groups is introduced in \cite{Randal-WilliamsWahl2017Homologicalstabilityautomorphism}).

In \S\S\ref{sTCS},\ref{sHeightDegree} below we will define a \emph{twisted coefficient system of degree $d$} for the sequence $\{ C_n(M,X) \}$ to be a functor from a certain category $\cB(M,X)$ to $\mathsf{Ab}$ satisfying a certain finiteness condition. To state the main result, it is enough to mention that it includes the data of a $\pi_1 C_n(M,X)$-module $T_n$ for each $n$, and that the stabilisation map induces a natural map

\begin{equation}\label{eStabMapOnTwistedHomology}
H_*(C_n(M,X);T_n) \longrightarrow H_*(C_{n+1}(M,X);T_{n+1}).
\end{equation}

The main result of this paper is the following:

\begin{athm}\label{tMain}
Under Assumption \ref{Assumption}, if $T$ is a twisted coefficient system for $\{C_n(M,X)\}$ of degree $d$, then the map \eqref{eStabMapOnTwistedHomology} is an isomorphism in degrees $*\leq\frac{n-d}{2}$, and is split-injective in all degrees.
\end{athm}

This is a generalisation of \cite[Theorem 4.3]{Betley2002Twistedhomologyof}, where twisted homological stability is proved for the symmetric groups $\{ \Sigma_n \}$, corresponding to the case $M=\bR^\infty$ and $X=*$. In fact, we also slightly strengthen Betley's result in the case of the symmetric groups. The category $\cB(\bR^\infty,*)$ is equivalent to the category $\text{FI}\sharp$ of finite sets and partially-defined injections. This is a subcategory of the category $\Gamma$ of finite pointed sets (viewed as the category of finite sets and partially-defined functions). Betley's result is stated for functors $T \colon \Gamma \to \ab$, whereas our result only requires $T$ to be defined on the subcategory $\text{FI}\sharp \subset \Gamma$.

\begin{rmk}[\emph{Split-injectivity}]
The split-injectivity statement of this theorem is fairly easy, and has essentially the same proof as in the untwisted case. It is proved separately in \S\ref{sInjectivity}, and its proof does not depend on the twisted coefficient system being of finite degree -- this assumption is only required for surjectivity in the stable range.
\end{rmk}

\begin{rmk}[\emph{When $\cdot 2$ is invertible}]\label{rImprovedRanges}
If $T\colon \cB(M,X)\to\ab$ is a twisted coefficient system of $\bZ[\frac12]$-modules, i.e.\ its image lies in the subcategory $\bZ[\frac12]\text{-}\mathsf{mod}$ of $\mathsf{Ab}$, then the stability range in Theorem \ref{tMain} can be improved to $*\leq n-d$, as long as $M$ is at least $3$-dimensional. When $M$ is a surface, a similar improvement is possible if $T$ is a \emph{rational} twisted coefficient system, i.e.\ its image lies in the subcategory $\mathsf{Vect}_{\bQ}$ of $\mathsf{Ab}$. The improved range in this case is $*\leq n-d$ when $M$ is non-orientable and $*<n-d$ when $M$ is orientable. This uses the improved homological stability ranges, for untwisted coefficients, obtained in \cite{Church2012Homologicalstabilityconfiguration, Randal-Williams2013Homologicalstabilityunordered, KupersMiller2015Improvedhomologicalstability, Knudsen2014Bettinumbersand}. See Remark \ref{rRationalRange} after the proof of Theorem \ref{tMain} in \S\ref{sProof}.
\end{rmk}

\begin{rmk}[\emph{Related results}]\label{rRelatedResults}
We summarise here some related twisted homological stability results that are not included in the statement of Theorem \ref{tMain}.

Theorem D of \cite{Randal-WilliamsWahl2017Homologicalstabilityautomorphism} proves homological stability for the braid groups $\beta_n$ with coefficients in any functor $\Ubeta \to \ab$ of finite degree, where $\Ubeta$ is a certain category with the braid groups as its automorphism groups. There is a functor $\Ubeta \to \cB(\bR^2,*)$, and precomposing with this functor preserves degree, so this extends Theorem \ref{tMain} in the case of braid groups $(M,X)=(\bR^2,*)$. (We note however that the statement in Theorem \ref{tMain} about split-injectivity in all degrees is not recovered by their theorem.) For example, the unreduced Burau representations fit into their setting. They also recover a result of Church and Farb \cite[Corollary 4.4]{ChurchFarb2013Representationtheoryand}, who prove twisted homological stability for the braid groups with coefficients in certain finite degree functors $\Ubeta \to \text{FI} \to \ab$ built out of irreducible representations of the symmetric groups. Chen \cite{Chen2017Homologybraidgroups} computes explicitly the homology $H_*(\beta_n;\text{Bur}^{\text{r}}_n)$ with coefficients in the reduced Burau representations over $\bC$, and directly reads off stability from his calculations. The reduced and unreduced Burau representations fit into a much more general family of braid group representations, called the \emph{Lawrence-Krammer-Bigelow representations}, which are discussed briefly in \S\ref{ss:examples-braid-groups} below.

In \cite[Theorem 3.4.1]{SamSnowden2014RepresentationscategoriesG}, Sam and Snowden prove that the sequence of groups $G \wr \Sigma_n$ is homologically stable with coefficients in any finite degree functor $\text{FI}_G \to R\text{-mod}$ if $G$ is polycyclic-by-finite. This extends Theorem \ref{tMain} in the case $(M,X)=(\bR^\infty,BG)$, since $\text{FI}_G$ embeds as a subcategory of $\cB(\bR^\infty,BG)$, and precomposition by this embedding preserves degree. Their methods are quite different to those of \cite{Randal-WilliamsWahl2017Homologicalstabilityautomorphism} but are in fact more analogous to ours, in that we both proceed by deducing twisted homological stability from untwisted homological stability. (They use methods of \cite{SamSnowden2017Groebnermethodsrepresentations} to prove that the category of representations of $\text{FI}_G$ is Noetherian, and then apply Theorem 4.2 of \cite{PutmanSam2014Representationstabilityand} and the known stability of $G \wr \Sigma_n$ with constant coefficients to deduce twisted stability.) Twisted homological stability with finite degree coefficients is in fact true for $G \wr \Sigma_n$ (and indeed also $G \wr \beta_n$) for any group $G$, by Theorem D of \cite{Randal-WilliamsWahl2017Homologicalstabilityautomorphism}.

Recently, Krannich \cite{Krannich2017Homologicalstabilitytopological} has extended the techniques of \cite{Randal-WilliamsWahl2017Homologicalstabilityautomorphism} to a topological setting, where one begins with an $\bN$-graded $E_1$-module over an $E_2$-algebra. Considering $C(\bR^2) = \coprod_n C_n(\bR^2)$ as a module over itself, he recovers Theorem D of \cite{Randal-WilliamsWahl2017Homologicalstabilityautomorphism} for the braid groups. Moreover, if $M$ is an open, connected $d$-manifold, considering $C(M,X)$ as a module over $C(\bR^d,X)$, he proves twisted homological stability for the configuration spaces $C_n(M,X)$ with coefficients in any finite degree functor defined on an analogue $\cC^X(M)$ of $\Ubeta$. An earlier version of this paper contained a conjecture about extending Theorem \ref{tMain} to coefficient systems defined only on a certain subcategory of $\cB(M,X)$ (see section \S\ref{ssDegree} for the definition of this subcategory). This has now been confirmed by the results of Krannich: there is a functor $\cC^X(M) \to \cB(M,X)$ whose image is the subcategory of the conjecture, and precomposition by this functor preserves the degree of coefficient systems.

Another family of coefficient systems (different from those appearing in Theorem \ref{tMain}) with respect to which twisted homological stability is known, is \emph{abelian coefficient systems}. See \S 5.6.2 of \cite{Randal-WilliamsWahl2017Homologicalstabilityautomorphism} for the case of the braid groups and Theorem D(i) of \cite{Krannich2017Homologicalstabilitytopological} for configuration spaces in general. A special case of the latter theorem (homological stability for configuration spaces with the abelian twisted coefficients $\bZ[\bZ/2]$) was proved earlier by the author in \cite{Palmer2013Homologicalstabilityoriented}.
\end{rmk}

\subsection{Stable twisted homology.}

After establishing homological stability for a sequence of spaces, the natural next step is to compute its limiting (or ``stable'') homology. For configuration spaces, and with constant coefficients, the answer is given by \cite{Segal1973Configurationspacesand} and \cite{McDuff1975Configurationspacesof}. It is also known for some particular non-constant twisted coefficient systems in the case of the braid groups: the reduced Burau representations \cite{Chen2017Homologybraidgroups} and the reduced and unreduced Coxeter representations \cite[\S I.5]{Vassiliev1992Complementsdiscriminantssmooth}. The latter computation is also recovered and extended in work in progress of Arthur Souli{\'e}. As far as the author is aware, no other computations yet exist for the stable homology of configuration spaces with non-constant twisted coefficients.

There is a general method for computing the stable twisted homology of a sequence of groups, introduced by Djament and Vespa \cite{DjamentVespa2010Surlhomologiedes, DjamentVespa2015Surlhomologiedes} and used by them for orthogonal and symplectic groups, and $\mathrm{Aut}(F_n)$. This may be adaptable to surface braid groups, but it is less likely to be applicable for configuration spaces on higher-dimensional manifolds, since these are not aspherical. Randal-Williams \cite{Randal-Williams2016Cohomologyautomorphismgroups} has a different, more topological approach to computing stable twisted homology, which he has applied to $\mathrm{Aut}(F_n)$ and to mapping class groups of surfaces, and which may be more easily adaptable to configuration spaces.

\subsection{Corollaries.}

Two special cases of Theorem \ref{tMain} are as follows. For the first, fix a principal ideal domain $R$ and a path-connected based space $Z$ with $H_*(Z;R)$ flat over $R$ in all degrees. For example we could take $R$ to be a field, or we could take $R=\bZ$ and assume that the integral homology of $Z$ is torsion-free. Also choose non-negative integers $q,h$ and suppose that $\widetilde{H}_*(Z;R)=0$ in the range $*\leq h$. The homology group $H_q(Z^n;R)$ is a $\bZ[\Sigma_n]$-module given by permuting the factors of $Z^n$, and hence also a $\bZ[\pi_1(C_n(M,X))]$-module via the projection $\pi_1(C_n(M,X)) \to \Sigma_n$.

\begin{acoro}\label{cHomology}
There are isomorphisms
\[
H_*\bigl(C_n(M,X);H_q(Z^n;R)\bigr) \;\cong\; H_*\bigl(C_{n+1}(M,X);H_q(Z^{n+1};R)\bigr)
\]
in the range $*\leq \frac12 \bigl( n-\bigl\lfloor \frac{q}{h+1} \bigr\rfloor \bigr)$. If we take $R=\bQ$, or $R$ is a ring in which $2$ is invertible and $M$ is at least $3$-dimensional, then this holds in the larger range $*\leq n-\bigl\lfloor \frac{q}{h+1} \bigr\rfloor$ \textup{(}except in the case where $M$ is an orientable surface, in which case the larger range is $*< n-\bigl\lfloor \frac{q}{h+1} \bigr\rfloor$\textup{)}.
\end{acoro}

\begin{rmk}[\emph{Configurations with twisted labels}]\label{rTwistedLabels}
A consequence of Corollary \ref{cHomology} is (untwisted) homological stability for configuration spaces $C_n(M,\pi)$ with labels in a fibre bundle $\pi\colon E\to M$ with path-connected fibres. Here, $C_n(M,\pi) = \{ (e_1,\ldots,e_n)\in E^n \;|\; \pi(e_i) \neq \pi(e_j) \text{ for } i\neq j \}/\Sigma_n$, generalising the notion of configuration spaces with labels in a fixed space $X$. This uses the Serre spectral sequence for the fibre bundle $C_n(M,\pi) \to C_n(M)$ that forgets the labels, and which has $E^2$ page isomorphic to the twisted homology groups of $C_n(M)$ with coefficients in the homology groups of $F^n$, where $F$ is the typical fibre of $\pi$. Corollary \ref{cHomology} says that the stabilisation maps induce a map of spectral sequences which is an isomorphism in a range on the $E^2$ page, as long as we take field coefficients. One can then reconstruct an integral homological stability result from the fields $\bF_p$ and $\bQ$. This is proved in more detail in Appendix B of \cite{CanteroPalmer2015homologicalstabilityconfiguration} and also in Appendix~A of \cite{KupersMiller2014Encellattachments}.

We note that this result can alternatively be proved using a generalisation of the proof of \cite{Randal-Williams2013Homologicalstabilityunordered}, which is concerned with configuration spaces with labels in a fixed space. This alternative proof is also sketched in Appendix A of \cite{KupersMiller2014Encellattachments}. Moreover, homological stability for $C_n(M,\pi)$, with twisted as well as constant coefficients, is implied by the recent work \cite{Krannich2017Homologicalstabilitytopological} of Krannich (\cf Remark \ref{rRelatedResults}).
\end{rmk}

Now we describe a second special case of Theorem \ref{tMain}.
For an ordered partition $\mu = (\mu_1,\ldots,\mu_k)$ of $\lvert \mu \rvert = \mu_1 + \cdots + \mu_k$, denote by $\Sigma_\mu$ the product of symmetric groups $\Sigma_{\mu_1} \times \cdots \times \Sigma_{\mu_k}$, which is naturally a subgroup of $\Sigma_{\lvert\mu\rvert}$. Fix an ordered partition $\lambda$, and assume that $n\geq \lvert \lambda \rvert$, so that there is an induced ordered partition $\lambda[n] \coloneqq (n-\lvert\lambda\rvert,\lambda_1,\ldots,\lambda_k)$ of $n$. Then $\Sigma_n/\Sigma_{\lambda[n]}$ is a (transitive) $\Sigma_n$-set. If $R$ is a ring, then $R[\Sigma_n/\Sigma_{\lambda[n]}]$ is a $\pi_1(C_n(M,X))$-module via the projection $\pi_1(C_n(M,X)) \to \Sigma_n$.

\begin{acoro}\label{cPartitions}
There are isomorphisms
\[
H_*\bigl(C_n(M,X); R \bigl[ \Sigma_n/\Sigma_{\lambda[n]} \bigr] \bigr) \;\cong\; H_*\bigl(C_{n+1}(M,X); R \bigl[ \Sigma_{n+1}/\Sigma_{\lambda[n+1]} \bigr] \bigr)
\]
in the range $*\leq \frac12 \bigl( n-\lvert \lambda \rvert \bigr)$. If we take $R=\bQ$, or $R$ is a ring in which $2$ is invertible and $M$ is at least $3$-dimensional, then this holds in the larger range $*\leq n-\lvert \lambda \rvert$ \textup{(}except in the case where $M$ is an orientable surface, in which case the larger range is $*< n-\lvert \lambda \rvert$\textup{)}.
\end{acoro}

In particular this includes stability for coefficients in $\bZ[\Sigma_n/\Sigma_{n-k}]$ or in $\bZ[\Sigma_n/(\Sigma_k \times \Sigma_{n-k})]$ in the range $*\leq\frac{n-k}{2}$ by taking $\lambda$ to be $(1,\ldots,1)$ or $(k)$ respectively.

\begin{rmk}[\emph{Coloured configuration spaces}]\label{rColoured}
Corollary \ref{cPartitions} may in fact be deduced quickly from untwisted homological stability, as follows. First note that
\[
H_*\bigl(C_n(M,X); R \bigl[ \Sigma_n/\Sigma_{\lambda[n]} \bigr] \bigr) \;\cong\; H_*(C_{\lambda[n]}(M,X);R),
\]
where the \emph{coloured configuration space} $C_{\lambda[n]}(M,X)$ is defined to be the covering space of $C_n(M,X)$ with $\bigl\lvert \Sigma_n/\Sigma_{\lambda[n]} \bigr\rvert = \binom{n}{\lambda_1}\binom{n-\lambda_1}{\lambda_2} \cdots \binom{n-\lambda_1 -\cdots -\lambda_{k-1}}{\lambda_k}$ sheets, in which the $n$ points are coloured according to the partition $\lambda[n]$. There is a stabilisation map
\[
C_{\lambda[n]}(M,X) \longrightarrow C_{\lambda[n+1]}(M,X)
\]
given by adding a point of the first colour to a coloured configuration (similarly to the stabilisation map defined in Definition \ref{dStabilisationMap}). This commutes up to homotopy with the projections to $C_\lambda(M,X)$, which are fibre bundles, and the map of fibres is the ordinary stabilisation map $C_{n-\lvert \lambda \rvert}(M_{\lvert \lambda \rvert},X) \to C_{n+1-\lvert \lambda \rvert}(M_{\lvert \lambda \rvert},X)$, where $M_{\lvert \lambda \rvert}$ denotes the manifold $M$ with $\lvert \lambda \rvert$ points removed. The result then follows by applying the relative Serre spectral sequence associated to this map of fibre bundles over $C_\lambda(M,X)$.
\end{rmk}

\begin{rmk}[\emph{Representation stability}]\label{rRepresentationStability}
Write $F_n(M,X)$ for the configuration space of $n$ ordered, distinct points in $M$ labelled by $X$. This may also be written $C_{(1,1,\ldots,1)}(M,X)$ in the notation of the previous remark and is an $(n!)$-sheeted covering space of $C_n(M,X)$. The sequence of graded $\bQ[\Sigma_n]$-modules $H^*(F_n(M,X);\bQ)$ is \emph{representation stable}, a notion introduced in \cite{ChurchFarb2013Representationtheoryand} and first proved in this case by \cite{Church2012Homologicalstabilityconfiguration}.

There is an argument of S{\o}ren Galatius, involving only the elementary representation theory of the symmetric groups, that proves representation stability for $H^*(F_n(M,X);\bQ)$ using, as an input, twisted homological stability for $C_n(M,X)$ with coefficients in $\bQ[\Sigma_n/\Sigma_{\lambda[n]}]$, which is a special case of Corollary \ref{cPartitions} above. This suggests an underlying connection between representation stability and twisted homological stability. We note that representation stability for $H^*(F_n(M,X);\bQ)$ may also be deduced from twisted homological stability for $C_n(M,X)$ with respect to a different twisted coefficient system than the one considered in Corollary \ref{cPartitions}: see Corollary 5.17 of \cite{Krannich2017Homologicalstabilitytopological} for the details. The twisted coefficient system used in that case does not fit into the setting of Theorem~\ref{tMain}.
\end{rmk}

\paragraph*{A note on terminology.}
To keep our terminology from becoming ambiguous, we will always use the terms ``local coefficient system'' and ``twisted coefficient system'' as follows. For a space $Y$, a \emph{local coefficient system} for $Y$ will have its usual meaning as a bundle of abelian groups over $Y$, or a functor from the fundamental groupoid of $Y$ to $\ab$, or (when $Y$ is based and path-connected) a $\pi_1(Y)$-module. The phrase \emph{twisted coefficient system} will always be used in the sense of Definition \ref{dTCS} below; in particular it applies to a \emph{sequence} of spaces.

\paragraph*{Acknowledgements.}
The content of this paper appeared, in a slightly different form, as part of the author's PhD thesis in 2013, and he would like to thank his supervisor, Ulrike Tillmann, for her invaluable advice and guidance throughout his PhD. He would also like to thank many other people for enlightening discussions: Cristina Anghel and Christian Blanchet (for discussions about the Lawrence representations of the braid groups), Aur{\'e}lien Djament, S{\o}ren Galatius (for sharing his proof of representation stability for ordered configuration spaces, \cf Remark \ref{rRepresentationStability}), Manuel Krannich, Oscar Randal-Williams and Arthur Souli{\'e}. Additionally, he would like to thank the anonymous referee, as well as Nathalie Wahl, for very helpful remarks on and corrections to the earlier drafts of this paper.


\section{Twisted coefficient systems}\label{sTCS}

\subsection{Setup.}\label{ssTCS-setup}
First we fix some data. Recall from Assumption \ref{Assumption} that $M$ is an open, connected manifold of dimension at least $2$ and $X$ is a path-connected space. This assumption on $M$ means that we may pick a connected manifold $\mbar$ with non-empty boundary $\partial\mbar$ whose interior is $M$ (although we must allow $\partial\mbar$ to be non-compact in general). Also choose a basepoint $x_0$ for $X$. Choose a point $a\in \partial\mbar$, and let $U$ be a coordinate neighbourhood of $a$ with an identification $U\cong \bR^d_+ = \{ x\in \bR^d \;|\; x_1\geq 0 \}$ which sends $a$ to $0$. Also choose a self-embedding $e\colon \mbar\emb\mbar$ which is isotopic to the identity, is \emph{equal} to the identity outside $U$, and such that $e(a)\in M$ (i.e.\ in the interior of $\mbar$). Moreover, we choose an isotopy $I\colon e\simeq \mathrm{id}_{\mbar}$. We obtain a sequence of points in $M$ by defining
\[
a_1 \coloneqq e(a) \qquad\qquad a_n \coloneqq e(a_{n-1}) \text{ for } n\geq 2.
\]
The isotopy $I$ provides us with canonical paths $p_n\colon [0,1]\to M$ between $a_n$ and $a_{n+1}$.

\subsection{The configuration space and the stabilisation map.}\label{ssTCS-stabmap}
Recall that the configuration space of $n$ unordered points in $M$ with labels in $X$ is defined to be
\[
C_n(M,X) \coloneqq ((M^n \smallsetminus \Delta)\times X^n)/\Sigma_n = (\Emb(n,M)\times X^n)/\Sigma_n,
\]
where $\Delta = \{(p_1,\ldots,p_n)\in M^n \;|\; p_i = p_j \text{ for some } i\neq j\}$ is the so-called \emph{fat diagonal} of $M^n$, and the symmetric group $\Sigma_n$ acts diagonally, permuting the points of $M$ along with their labels in $X$. Thus a labelled configuration is an unordered set of ordered pairs in $M\times X$, generically denoted by $\{(p_1,x_1),\ldots,(p_n,x_n)\}$. When $X$ is a point we will also write $C_n(M) = C_n(M,X)$.

\begin{defn}\label{dStabilisationMap}
The \emph{stabilisation map} $s_n\colon C_n(M,X) \to C_{n+1}(M,X)$ is defined by
\[
\{(p_1,x_1),\ldots,(p_n,x_n)\} \;\mapsto\; \{(e(p_1),x_1),\ldots,(e(p_n),x_n),(a_1,x_0)\}.
\]
Essentially, the existing configuration is ``pushed'' further into the interior of the manifold by $e$, and the new configuration point $a_1$ added in the newly vacated space. Up to homotopy, the only ``extra data'' that this map depends on is the component of $\partial\mbar$ containing $a$.
\end{defn}

\subsection{Twisted coefficient systems.}\label{ssTCS-TCS}
We define the category $\cB(M,X)$ to have the non-negative integers as its objects, and a morphism $m \to n$ is a choice of $k\leq \mathrm{min}\{m,n\}$ and a path in $C_k(M,X)$ from a $k$-element subset of $\{(a_1,x_0),\ldots,(a_m,x_0)\}$ to a $k$-element subset of $\{(a_1,x_0),\ldots,(a_n,x_0)\}$ up to endpoint-preserving homotopy. The identity is given by $k=m=n$ and the constant path. Composition of two morphisms is given by concatenating paths and deleting configuration points for which the concatenated path is defined only half-way. For example (omitting the labels in $X$):
\begin{equation}\label{eComposition}
\centering
\begin{split}
\begin{tikzpicture}
[x=1mm,y=1mm]
\node (al1) at (0,0) [fill,circle,inner sep=1pt] {};
\node (al2) at (0,2) [fill,circle,inner sep=1pt] {};
\node (al3) at (0,4) [fill,circle,inner sep=1pt] {};
\node (ar1) at (10,0) [fill,circle,inner sep=1pt] {};
\node (ar2) at (10,2) [fill,circle,inner sep=1pt] {};
\node (ar3) at (10,4) [fill,circle,inner sep=1pt] {};
\node (ar4) at (10,6) [fill,circle,inner sep=1pt] {};
\node (ar5) at (10,8) [fill,circle,inner sep=1pt] {};
\draw (al1) .. controls (5,0) and (5,4) .. (ar3);
\draw[white,line width=1mm] (al3) .. controls (5,4) and (5,0) .. (ar1);
\draw (al3) .. controls (5,4) and (5,0) .. (ar1);
\draw[white,line width=1mm] (al2) .. controls (5,2) and (5,8) .. (ar5);
\draw (al2) .. controls (5,2) and (5,8) .. (ar5);
\node at (13,4) {$\circ$};
\begin{scope}[xshift=16mm]
\node (bl1) at (0,0) [fill,circle,inner sep=1pt] {};
\node (bl2) at (0,2) [fill,circle,inner sep=1pt] {};
\node (bl3) at (0,4) [fill,circle,inner sep=1pt] {};
\node (bl4) at (0,6) [fill,circle,inner sep=1pt] {};
\node (bl5) at (0,8) [fill,circle,inner sep=1pt] {};
\node (br1) at (10,0) [fill,circle,inner sep=1pt] {};
\node (br2) at (10,2) [fill,circle,inner sep=1pt] {};
\node (br3) at (10,4) [fill,circle,inner sep=1pt] {};
\node (br4) at (10,6) [fill,circle,inner sep=1pt] {};
\draw (bl1) .. controls (5,0) and (5,2) .. (br2);
\draw (bl2) .. controls (5,2) and (5,6) .. (br4);
\draw[white,line width=1mm] (bl5) .. controls (5,8) and (5,4) .. (br3);
\draw (bl5) .. controls (5,8) and (5,4) .. (br3);
\end{scope}
\node at (32,2.5) {$=$};
\begin{scope}[xshift=38mm]
\node (cl1) at (0,0) [fill,circle,inner sep=1pt] {};
\node (cl2) at (0,2) [fill,circle,inner sep=1pt] {};
\node (cl3) at (0,4) [fill,circle,inner sep=1pt] {};
\node (cr1) at (10,0) [fill,circle,inner sep=1pt] {};
\node (cr2) at (10,2) [fill,circle,inner sep=1pt] {};
\node (cr3) at (10,4) [fill,circle,inner sep=1pt] {};
\node (cr4) at (10,6) [fill,circle,inner sep=1pt] {};
\draw (cl3) .. controls (5,4) and (5,2) .. (cr2);
\draw[white,line width=1mm] (cl2) .. controls (5,2) and (5,4) .. (cr3);
\draw (cl2) .. controls (5,2) and (5,4) .. (cr3);
\end{scope}
\end{tikzpicture}
\end{split}
\end{equation}
When $X$ is a point we will also write $\cB(M) = \cB(M,X)$. This is the \emph{partial braid category on $M$}.

\begin{defn}\label{dTCS}
A twisted coefficient system, associated to the direct system of spaces $\{ C_n(M,X) \}$, is a functor from $\cB(M,X)$ to the category $\ab$ of abelian groups.
\end{defn}

\begin{rmk}\label{rmk:when-B-is-symmetric-etc}
If the manifold $M$ splits as $M \cong \bR \times M^\prime$, then there is a monoidal structure on $\cB(M,X)$ (depending on the choice of such a splitting), given, intuitively, by placing two braids (i.e.\ paths of configurations) side by side in the $\bR$ direction -- in the figure above it corresponds to stacking braids vertically. If $M^\prime$ splits further as $M^\prime \cong \bR \times M^{\prime\prime}$, then a choice of such a splitting induces a braiding for this monoidal structure. Moreover, if $M^{\prime\prime}$ splits again as $M^{\prime\prime} \cong \bR \times M^{\prime\prime\prime}$, then this braiding is symmetric.

However, this is not the key structure that we use. Instead, we use (a) an endofunctor $s$ on $\cB(M,X)$\footnote{\Cf the stabilisation map (Definition \ref{dStabilisationMap}).} and a natural transformation $\iota \colon \mathrm{id} \to s$,\footnote{\Cf the morphisms $\iota_n$ defined immediately below and illustrated in \eqref{eIota}.} which together lead to the notion of the \emph{degree} of a functor $\cB(M,X) \to \mathsf{Ab}$, and (b) the existence of morphisms in $\cB(M,X)$ that ``forget'' points. More formally, the latter says that $\cB(M,X)$ has a subcategory (consisting of all constant braids) isomorphic to $\cI$, which is the category with non-negative integers as objects and where $\cI(m,n)$ is the power set of $\{1,\ldots,\mathrm{min}(m,n)\}$, with composition given by intersection. This latter structure leads to the notion of the \emph{height} of a functor $\cB(M,X) \to \mathsf{Ab}$.

There is an interaction between the structures (a) and (b) and the monoidal structure (when it exists), which may also be used to define notions of \emph{degree} and \emph{height} for functors $\cB(M,X) \to \mathsf{Ab}$. See \S 2 (especially \S 2.3) and \S 3 (especially \S 3.12) of \cite{Palmer2017comparisontwistedcoefficient} for a discussion of this interaction. See also Remark \ref{rmk:comparision-to-RWW} below.
\end{rmk}

\paragraph*{Further structure.} We now explain how the stabilisation map induces a map between homology groups of configuration spaces twisted by a functor $T \colon \cB(M,X) \to \ab$.

For each $n$, take $\{(a_1,x_0),\ldots,(a_n,x_0)\}$ as the basepoint of $C_n(M,X)$. Then the automorphism group of the object $n$ of $\cB(M,X)$ is precisely the fundamental group $\pi_1 C_n(M,X)$. So if we are given a functor $T\colon \cB(M,X)\to\ab$ this induces an action of $\pi_1 C_n(M,X)$ on $T_n \coloneqq T(n)$, and we can define the local homology $H_*(C_n(M,X);T_n)$.

For every object $n$ of $\cB(M,X)$ there is a natural morphism $\iota_n \colon n \to n+1$ represented by the path in $C_n(M,X)$ from $\{ (a_1,x_0),\ldots,(a_n,x_0) \}$ to $\{ (a_2,x_0),\ldots,(a_{n+1},x_0) \}$ where each configuration point $a_i$ travels along the path $p_i$ (see \S\ref{ssTCS-setup}) and the labels $x_0$ stay constant. Schematically, this may be pictured as:

\begin{equation}\label{eIota}
\centering
\begin{split}
\begin{tikzpicture}
[x=1mm,y=1mm]
\node (al1) at (0,0) [fill,circle,inner sep=1pt] {};
\node (al2) at (0,2) [fill,circle,inner sep=1pt] {};
\node (al3) at (0,4) [fill,circle,inner sep=1pt] {};
\node (al4) at (0,6) [fill,circle,inner sep=1pt] {};
\node (al5) at (0,8) [fill,circle,inner sep=1pt] {};
\node (ar1) at (20,0) [fill,circle,inner sep=1pt] {};
\node (ar2) at (20,2) [fill,circle,inner sep=1pt] {};
\node (ar3) at (20,4) [fill,circle,inner sep=1pt] {};
\node (ar4) at (20,6) [fill,circle,inner sep=1pt] {};
\node (ar5) at (20,8) [fill,circle,inner sep=1pt] {};
\node (ar6) at (20,10) [fill,circle,inner sep=1pt] {};
\draw (al1) -- (ar2);
\draw (al2) -- (ar3);
\draw (al5) -- (ar6);
\node at (10,7) {$\vdots$};
\node at (0,0) [anchor=east,font=\scriptsize] {$a_1$};
\node at (0,2) [anchor=east,font=\scriptsize] {$a_2$};
\node at (0,8) [anchor=east,font=\scriptsize] {$a_n$};
\node at (20,0) [anchor=west,font=\scriptsize] {$a_1$};
\node at (20,2) [anchor=west,font=\scriptsize] {$a_2$};
\node at (20,8) [anchor=west,font=\scriptsize] {$a_n$};
\node at (20,10) [anchor=west,font=\scriptsize] {$a_{n+1}$};
\node at (10,1.5) [anchor=north,font=\scriptsize] {$x_0$};
\node at (10,8.5) [anchor=south,font=\scriptsize] {$x_0$};
\end{tikzpicture}
\end{split}
\end{equation}

For any $\gamma\in \pi_1 C_n(M,X) = \Aut_{\cB(M,X)}(n)$ it is easy to check that
\[
\iota_n \circ \gamma = (s_n)_*(\gamma) \circ \iota_n,
\]
so for any $T$ the map $T\iota_n\colon T_n\to T_{n+1}$ is equivariant with respect to the group homomorphism $(s_n)_*\colon \pi_1 C_n(M,X)\to \pi_1 C_{n+1}(M,X)$. Hence we have an induced map
\[
(s_n;T\iota_n)_* \colon H_*(C_n(M,X);T_n) \to H_*(C_{n+1}(M,X);T_{n+1}).
\]
This is the map \eqref{eStabMapOnTwistedHomology} which induces the isomorphism in Theorem \ref{tMain}.

\begin{notation}
From now on, by abuse of notation, we will denote the induced map $T\iota_n\colon T_n\to T_{n+1}$ also by $\iota_n\colon T_n\to T_{n+1}$. Similarly for the left-inverse $\pi_n \colon n+1 \to n$ of $\iota_n$ (represented by the reverse of the path that represents $\iota_n$; see \S\ref{ssDegree}): we denote its image under $T$ also by $\pi_n\colon T_{n+1}\to T_n$.
\end{notation}

\subsection{A special case.}\label{ssTCS-specialcase}
Let $X$ be a point and assume that $M$ is simply-connected and of dimension at least $3$. These conditions imply that $\pi_1 C_n(M) \cong \Sigma_n$, in other words a path in $C_n(M)$ from the basepoint $\{a_1,\ldots,a_n\}$ to itself is determined by the permutation it induces on the set $\{a_1,\ldots,a_n\}$. More generally, any morphism $m \to n$ in $\cB(M)$ is determined by the partially-defined injection $\{a_1,\ldots,a_m\} \dashrightarrow \{a_1,\ldots,a_n\}$ that it induces. Hence there is a canonical isomorphism of categories $\cB(M)\cong \Sigma$, where $\Sigma$ is the category defined as follows.

\begin{defn}\label{dSigma}
The category $\Sigma$ has objects $\{0,1,2,\ldots\}$, and a morphism from $m$ to $n$ in $\Sigma$ is a partially-defined injection $m\dashrightarrow n$. Composition is then composition of partially-defined functions (where the composite function is defined exactly where it is possible to define it). Note that $\Sigma$ is an \emph{inverse category}, i.e.\ every morphism $f$ has a unique morphism $g$ such that $fgf=f$ and $gfg=g$ (this seems to have been first defined in \cite{Kastl1979Inversecategories}; see also \S 2 of \cite{Linckelmann2013inversecategoriesand}). It is a subcategory of the category with objects $\{0,1,2,\ldots\}$ and morphisms \emph{all} partially-defined functions (not necessarily injective), which is precisely $\Gamma^{\mathrm{op}}$, a skeleton of the category $\mathsf{Set}_*^{\mathsf{fin}}$ of finite pointed sets. Partially-defined injections are also sometimes called partially-defined bijections. The category $\Sigma$ also has other names in the literature, including $\mathbf{finPInj}$ \cite{Heunen2009Categoricalquantummodels}, $\text{FI}\sharp$ \cite{ChurchEllenbergFarb2015FImodulesstability}, $\Theta$ \cite{CollinetDjamentGriffin2013Stabilitehomologiquepour} and $\widetilde{\Theta}$ \cite{DjamentVespa2013Foncteursfaiblementpolynomiaux}.
\end{defn}

In particular we have $\cB(\bR^\infty) \cong \Sigma$. Of course, $\bR^\infty$ is not a finite-dimensional manifold, as was assumed of $M$, but the definitions make sense for arbitrary spaces $M$ and $X$, and $C_n(\bR^\infty)$ is the colimit of the spaces $C_n(\bR^d)$ under the obvious inclusions. The space $\mathrm{Emb}(n,\bR^\infty)$ is a contractible Hausdorff space on which the natural action of $\Sigma_n$ is free and properly discontinuous, so its quotient $C_n(\bR^\infty)$ is a model for the classifying space $B\Sigma_n$.

For any $M$ and $X$, there is a functor $\cB(M,X) \to \Sigma$ given by forgetting both the labels of the paths and the paths themselves, remembering only the partially-defined injection induced by the paths. This means that any twisted coefficient system $\Sigma\to\ab$ canonically induces a twisted coefficient system $\cB(M,X)\to \Sigma\to\ab$ (\cf Remark 4.6 of \cite{Palmer2017comparisontwistedcoefficient}).

\subsection{A functorial viewpoint.}\label{ssTCS-functorial}

This functor $\cB(M,X) \to \Sigma$ arises naturally in another way, if we view $\cB$ as a functor of $M$ and $X$. More precisely, we think of $\cB(-,-)$ as a functor taking as input a based space $X$ and a manifold $\mbar$ equipped with a collar neighbourhood together with a basepoint on $\partial\mbar$. (Morphisms of such data are based maps $X \to Y$ together with based embeddings $\mbar \hookrightarrow \nbar$ that are \emph{neat}, i.e., compatible with the collar neighbourhoods.) Its output is the category $\cB(M,X)$ equipped with certain additional structure. See \S 4 of \cite{Palmer2017comparisontwistedcoefficient} for more precise details. Now, by the Whitney Embedding Theorem,\footnote{\label{f:Whitney} The Whitney Embedding Theorem implies that any (paracompact) smooth manifold without boundary admits an embedding into some Euclidean space. One may deduce from this the appropriate analogous fact for manifolds with collared boundary -- see Lemma A.1 of \cite{Palmer2017comparisontwistedcoefficient} for the precise statement.} any such $\mbar$ admits a neat embedding into some Euclidean halfspace $\bR^N_{+}$. This embedding, together with the trivial map from $X$ to a point, induces a functor $\cB(M,X) \to \cB(\mathrm{int}(\bR^N_{+})) \cong \Sigma$, which is isomorphic to the forgetful functor described above.

\subsection{A more general case.}\label{ssTCS-moregeneral}
Instead of configurations of points (closed $0$-dimensional submanifolds), one may consider configurations of closed submanifolds of higher dimension. Let $\mbar$ be a connected manifold with non-empty boundary and of dimension at least $2$, as before. Also fix a closed manifold $P$ and an embedding $\iota_0 \colon P\hookrightarrow \partial\mbar$. Choose an embedding $e\colon \mbar \hookrightarrow \mbar$ which is isotopic to the identity and such that $e(\mbar)$ is disjoint from $\iota_0(P)$. We obtain a sequence of pairwise-disjoint embeddings of $P$ into $M$ by defining $\iota_n \coloneqq e^n \circ \iota_0$. Writing the disjoint union $P\sqcup \cdots \sqcup P$ of $n$ copies of $P$ as $nP$ for short, define $C_{nP}(M)$ to be the path-component of $\mathrm{Emb}(nP,M)/\mathrm{Diff}(nP)$ containing $[\iota_1 \sqcup \cdots \sqcup \iota_n]$. A stabilisation map $C_{nP}(M)\to C_{(n+1)P}(M)$ may then be defined by sending $[\phi_1 \sqcup \cdots \sqcup \phi_n]$ to $[(e\circ \phi_1) \sqcup \cdots \sqcup (e\circ \phi_n) \sqcup \iota_1]$. One may also define more complicated versions of this setup, in which the submanifolds in $C_{nP}(M)$ are parametrised modulo a subgroup of $\Diff(P)$ and come equipped with labels in some bundle over $\Emb(P,M)$.

Everything in this paper generalises to this setting, including an analogous notion of \emph{twisted coefficient system} for $\{C_{nP}(M)\}$, and the \emph{height} and \emph{degree} (see \S\ref{sHeightDegree}) of such a twisted coefficient system. In an article in preparation \cite{Palmer2018Homologicalstabilityconfiguration} we prove (untwisted) homological stability for these more general kinds of configuration spaces, as long as $\dim(P)\leq \frac12(\dim(M)-3)$. The arguments of this paper then immediately imply a twisted homological stability result for these spaces too.


\section{Height and degree of a twisted coefficient system}\label{sHeightDegree}

\subsection{Degree.}\label{ssDegree}
First we will define the \emph{degree} of a functor $T\colon \cB(M,X)\to\ab$. Recall from \S\ref{ssTCS-TCS} the natural morphisms $\iota_n \colon n \to n+1$. The adjective ``natural'' suggests that they should form a natural transformation, and in fact they do. For every morphism $\phi \colon m \to n$ of $\cB(M,X)$ we have a commutative square

\begin{equation}\label{eIotaSquare}
\centering
\begin{split}
\begin{tikzpicture}
[x=1.2mm,y=1.2mm,>=stealth']
\node (tl) at (0,10) {$m$};
\node (tr) at (20,10) {$m+1$};
\node (bl) at (0,0) {$n$};
\node (br) at (20,0) {$n+1$};
\draw[->] (tl) to node[above,font=\small]{$\iota_m$} (tr);
\draw[->] (bl) to node[above,font=\small]{$\iota_n$} (br);
\draw[->] (tl) to node[left,font=\small]{$\phi$} (bl);
\draw[->] (tr) to node[right,font=\small]{$S\phi$} (br);
\end{tikzpicture}
\end{split}
\end{equation}

where the morphism $S\phi$ is defined as follows: if $\phi$ is represented by a path $p$ in $C_k(M,X)$ for some $k\leq \mathrm{min}\{m,n\}$, then $S\phi$ is represented by the path $s_k\circ p$ in $C_{k+1}(M,X)$, where $s_k$ is the stabilisation map from \S\ref{ssTCS-stabmap}. Thus we have an endofunctor $S\colon \cB(M,X)\to \cB(M,X)$ (which we call the \emph{stabilisation endofunctor}) and a natural transformation $\iota\colon \mathrm{id}\Rightarrow S$. Note that each $\iota_n$ has an obvious left-inverse $\pi_n$, using the reverse of the path used to define $\iota_n$, and these morphisms fit together to form a left-inverse $\pi\colon S\Rightarrow \mathrm{id}$ for $\iota$.

So, given any $T\colon \cB(M,X)\to\ab$ we get a natural transformation $T\circ \iota\colon T\Rightarrow T\circ S$, or in other words a morphism in the abelian category $\ab^{\cB(M,X)}$. Denote its cokernel by $\Delta T\colon \cB(M,X)\to\ab$.

\begin{defn}\label{dDegree}
The \emph{degree} of a functor $T\colon \cB(M,X)\to\ab$ is defined recursively by
\[
\deg(0) = -1 \qquad\qquad \deg(T) = \deg(\Delta T) + 1,
\]
where $0$ is the identically-zero functor.
\end{defn}

\paragraph*{Restriction to fully-defined braids.}
The degree of $T$ in fact only depends on its restriction to the \emph{injective braid category}
\[
\cBf(M,X) \subset \cB(M,X)
\]
whose objects are the non-negative integers, just as for $\cB(M,X)$, and whose morphisms are ``braids'' in $M \times [0,1]$ from $\{a_1,\ldots,a_m\} \times \{0\}$ to $\{a_1,\ldots,a_n\} \times \{1\}$ (whose strands are labelled by $\Omega X$) with precisely $m$ strands -- in other words, the \emph{\textbf{f}ully-defined braids}, whereas morphisms in $\cB(M,X)$ are \emph{partially-defined braids}. Precisely, recall that a morphism in $\cB(M,X)$ is a choice of $k\leq\mathrm{min}\{m,n\}$ and a certain path in $C_k(M,X)$. This morphism belongs to $\cBf(M,X)$ if and only if $k=m$. Note that there are no morphisms from $m$ to $n$ if $m>n$. This means, in particular, that the object $0$ -- which is both initial and terminal in $\cB(M,X)$ -- fails to be terminal in the subcategory $\cBf(M,X)$, although it is still initial.

The stabilisation endofunctor $S\colon \cB(M,X) \to \cB(M,X)$ restricts to an endofunctor $S_{\mathsf{f}}$ on this subcategory $\cBf(M,X)$ and the natural transformation $\iota \colon \mathrm{id} \Rightarrow S$ restricts to $\iota_{\mathsf{f}} \colon \mathrm{id} \Rightarrow S_{\mathsf{f}}$. This does not in general have a left-inverse, so both functors $\coker(T\circ \iota_{\mathsf{f}})$ and $\kernel(T\circ \iota_{\mathsf{f}})$ may be non-trivial for a given functor $T\colon \cBf(M,X) \to \ab$.

\begin{defn}\label{dDegree2}
The zero functor $\cBf(M,X) \to \ab$ has degree $-1$. A non-zero functor $T\colon \cBf(M,X) \to \ab$ has degree $\leq d$ if and only if $\kernel(T\circ \iota_{\mathsf{f}})=0$ and $\degree(\coker(T\circ \iota_{\mathsf{f}}))\leq d-1$.
\end{defn}

This is called the \emph{injective degree} in \S 2 of \cite{Palmer2017comparisontwistedcoefficient} (see Definition 2.1), where we compare various related notions of degree. We note that $T\colon \cB(M,X) \to \ab$ has degree $d$ if and only if its restriction to $\cBf(M,X)$ has degree $d$ (this is easy to prove by induction on $d$).

\begin{rmk}\label{rmk:comparision-to-RWW}
If $M$ is of the form $\bR^2 \times M^\prime$ then, by Remark \ref{rmk:when-B-is-symmetric-etc}, $\cB(M,X)$ has a braided monoidal structure, hence in particular a so-called \emph{pre-braided} structure, so Definition 4.10 of \cite{Randal-WilliamsWahl2017Homologicalstabilityautomorphism} applies (if we take $A=0$ and $X=1$) and assigns a ``degree at $N$'' to a functor $\cB(M,X) \to \ab$ for each $N \in \bZ$. When $N=0$ this agrees with Definition \ref{dDegree}. Similarly, $\cBf(M,X)$ also has a pre-braided structure when $M = \bR^2 \times M^\prime$, so functors $\cBf(M,X) \to \ab$ are assigned a ``degree at $N$'' by \cite{Randal-WilliamsWahl2017Homologicalstabilityautomorphism}, which agrees with Definition \ref{dDegree2} when $N=0$. For a further discussion of the relation to the twisted coefficient systems of \cite{Randal-WilliamsWahl2017Homologicalstabilityautomorphism}, see \S 2.4 of \cite{Palmer2017comparisontwistedcoefficient}.
\end{rmk}

\begin{rmk}
Clearly, any constant functor $\cB(M,X) \to \ab$ has degree $\leq 0$, and therefore so does any functor isomorphic to a constant functor. Conversely, any functor of degree $\leq 0$ is isomorphic to a constant functor. One may see this as follows. Suppose that $\mathrm{deg}(T)\leq 0$. By definition of the degree, $T(\iota_n)$ is an isomorphism for all $n$, which, due to the structure of the category $\cB(M,X)$, implies that every morphism is sent to an isomorphism by $T$. Thus $T$ factors through the Grothendieck groupoid $\cG(\cB(M,X))$ of $\cB(M,X)$, which may be defined as the fundamental groupoid of its classifying space. But $\cB(M,X)$ has an initial object, so its classifying space is contractible and its Grothendieck groupoid is equivalent to the trivial category. Thus, up to isomorphism, $T$ factors through the trivial category.

Similarly, a functor $T \colon \cBf(M,X) \to \ab$ has degree $\leq 0$ if and only if it is isomorphic to a constant functor, by the same argument as above.
\end{rmk}

\subsection{Height.}\label{ssHeight}
Denote by $u$ the homomorphism $\pi_1 C_n(M,X)\to \Sigma_n$ which only remembers the permutation of the basepoint configuration (this is part of the canonical functor $\cB(M,X)\to \Sigma$ from \S\ref{ssTCS-specialcase}). Write $G_n \coloneqq \pi_1 C_n(M,X)$ and define $G_n^k \coloneqq u^{-1}(\Sigma_{n-k}\times \Sigma_k)$. To define the \emph{height} of a functor $T\colon \cB(M,X)\to\ab$ we need the following decomposition result:

\begin{prop}\label{pDecomp}
Let $T\colon \cB(M,X)\to\ab$ be any functor, and recall that we write $T_n = T(n)$. Then for $k=0, \ldots, n$ there is a direct summand \textup{(}as abelian groups\textup{)} $T_n^k$ of $T_n$ such that the action of $G_n^k\leq G_n$ on $T_n$ preserves it\textup{:} so it is also a direct summand as a $\bZ G_n^k$-module. Moreover, there is a decomposition of $T_n$ as a $\bZ G_n$-module\textup{:}
\begin{equation}\label{eDecomp}
T_n \;\cong\; \bigoplus_{k=0}^n \left( \bZ G_n \otimes_{\bZ G_n^k} T_n^k \right) .
\end{equation}
This identification is natural in the sense that $\iota_n \colon T_n \to T_{n+1}$ sends $T_n^k$ into $T_{n+1}^k$, and the map of the right-hand side induced by $\iota_n$ and $(s_n)_*$ corresponds under \eqref{eDecomp} to $\iota_n$ on the left-hand side.
\end{prop}

\begin{rmk}[\emph{A clarification}]
It is important to add that the $\bZ G_n^k$-submodules $T_n^k$ of $T_n$ in the above proposition are \emph{canonical}, that is to say that there is an operation that takes a functor $T\colon \cB(M,X)\to\ab$ as input and outputs a choice of such a submodule for each $k=0,\ldots,n$. This operation is given explicitly below (on the line immediately above Remark \ref{rObservations}). Thus when we speak of ``the'' module $T_n^k$ there is no ambiguity. Precisely \emph{how} this operation is defined is less important than the fact that it \emph{is} (well-)defined, so its definition is relegated to the proof of the proposition given below.

See also \S 4.5 of \cite{Palmer2017comparisontwistedcoefficient}, where a definition of the height of a twisted coefficient system is given in a more general setting, which includes an operation taking $T$ to $T^\prime$ (where $T^\prime$ encapsulates the data of the various $T_n^k$).
\end{rmk}

\begin{rmk}[\emph{Related decompositions}]
This is similar to the cross-effect decomposition of a functor from a pointed monoidal category (a monoidal category whose unit object is also initial and terminal) to an abelian category, which appears in \cite[Proposition 3.4]{HartlPirashviliVespa2015Polynomialfunctorsalgebras} (see also \cite[Proposition 2.11]{DjamentVespa2013Foncteursfaiblementpolynomiaux}), and the idea of which goes back to Eilenberg and MacLane~\cite[\S 9]{EilenbergMac1954groupsHnII}. However, our category $\cB(M,X)$ is not in general monoidal (it is when $M$ is of the form $\bR\times N$), so this setup does not cover our situation. A similar cross-effect decomposition appears in \cite[Proposition 1.4]{HartlVespa2011Quadraticfunctorspointed} for functors from a source category which has finite coproducts -- however, $\cB(M,X)$ also does not have finite coproducts. Yet another similar decomposition appears in \cite[Lemme 2.7(3)]{CollinetDjamentGriffin2013Stabilitehomologiquepour} for functors from a source category which is a wreath product $\cC \wr \Lambda$, where $\cC$ is any category and $\Sigma \leq \Lambda \leq \Se$. Here, $\Se$ is the category of finite sets and partially-defined functions and $\Sigma$ is its subcategory of partially-defined injections, as in Definition \ref{dSigma}. Our category $\cB(M,X)$ may be written as a wreath product $\pi_1(X,x_0) \wr \cB(M)$, where the wreath product is defined using the projection $\cB(M)\to\Sigma$. This is however not of the form considered in \cite{CollinetDjamentGriffin2013Stabilitehomologiquepour}, unless $M$ is simply-connected and of dimension at least $3$ (see \S\ref{ssTCS-specialcase}).
\end{rmk}

Since none of the existing decompositions in the literature covers the general case that we require, we give a complete proof of the decomposition \eqref{eDecomp} in our situation (i.e.\ Proposition \ref{pDecomp}). This is a little technical, so the reader may wish to skip directly to Definition \ref{dHeight} at this point. Before embarking upon the proof of Proposition \ref{pDecomp}, we point out a correction.

\begin{rmk}[\emph{A correction}]\label{rCorrection}
We should mention that the proof of the decomposition in Lemme 2.7(3) of \cite{CollinetDjamentGriffin2013Stabilitehomologiquepour} contains an error. We will briefly explain the error and sketch a corrected proof of their decomposition. See \cite[\S 2.1]{CollinetDjamentGriffin2013Stabilitehomologiquepour} for any unexplained notation. The first part of their proof establishes a decomposition
\begin{equation}\label{eCDG}
T(C) = \bigoplus_{P\subseteq \cP(E)} \bigcap_{A\in P} T_{A,M}(C),
\end{equation}
where $T_{A,S}(C)$ is defined to be $\kernel(T(d_{C,A})) \cap \image(T(d_{C,S}))$, $M = M_P$ is defined to be $\bigcap (\cP(E)\smallsetminus P)$ and the notation $\cP(E)$ means the power set of $E$.\footnote{There is a typo in \cite{CollinetDjamentGriffin2013Stabilitehomologiquepour}, where $M$ is incorrectly defined to be $\bigcap P$, rather than $\bigcap (\cP(E)\smallsetminus P)$.} The aim is then to show that this is equal to
\begin{equation}\label{eCDG2}
\bigoplus_{S\subseteq E} \bigcap_{A\in Q_S} T_{A,S}(C),
\end{equation}
where we define $Q_S \coloneqq \{ A\in \cP(E) \mid A\subsetneq S \}$. Define also $R_S \coloneqq \{ A\in \cP(E) \mid A \not\supseteq S \}$ and note that $Q_S \subseteq R_S$ with equality exactly when $S=E$. They state that $T_{A,S}(C)=0$ whenever $A\not\in Q_S$, but in fact this is only true under the stronger assumption that $A\not\in R_S$. We may therefore restrict the direct sum in \eqref{eCDG} to those $P$ such that $P\subseteq R_{M_P}$ (rather than $P\subseteq Q_{M_P}$, as claimed). The $P$ with this property are precisely the subsets $R_S$ for $S\subseteq E$. Moreover, the function $R\colon \cP(E) \to \cP(\cP(E))$ given by $S \mapsto R_S$ is injective (in contrast to the function $Q$), so we see that \eqref{eCDG} is equal to
\begin{equation}\label{eCDG3}
\bigoplus_{S\subseteq E} \bigcap_{A\in R_S} T_{A,S}(C).
\end{equation}
The final step of the proof is to show that restricting each intersection to the subset $Q_S$ of $R_S$ does not change it. The subset $Q_S$ is coinitial in $R_S$, but the function $\cP(E) \to \cP(T(C))$ given by $A \mapsto T_{A,S}(C)$ is non-\emph{increasing}, so this does not help us. Instead, this follows from the facts that $T_{A,S}(C) = T_{A\cap S,S}(C)$ and $\{ A\cap S \mid A\in R_S \} = Q_S$.

An alternative correction to the proof of Lemme 2.7(3) of \cite{CollinetDjamentGriffin2013Stabilitehomologiquepour} was pointed out to us later by Aur{\'e}lien Djament, which we also briefly sketch. The decomposition \eqref{eCDG} arises from the family of pairwise-commuting idempotents $\{ T(d_{C,S}) \mid S\subseteq E \}$ of $T(C)$. If we instead consider the subfamily $\{ T(d_{C,E\smallsetminus\{s\}}) \mid s\in E \}$, the corresponding decomposition is
\begin{equation}\label{eCDG4}
T(C) = \bigoplus_{S\subseteq E} \bigcap_{s\in S} T_{E\smallsetminus\{s\},S}(C).
\end{equation}
Using the fact that $T_{A,S}(C) = T_{A\cap S,S}(C)$, we may replace $T_{E\smallsetminus\{s\},S}(C)$ with $T_{S\smallsetminus\{s\},S}(C)$ on the right-hand side. Note that $\{ S\smallsetminus\{s\} \mid s\in S \}$ is cofinal in $Q_S$ and the function $\cP(E) \to \cP(T(C))$ given by $A \mapsto T_{A,S}(C)$ is non-increasing, so this is equal to \eqref{eCDG2}.
\end{rmk}

We now prove Proposition \ref{pDecomp}, for which we will need the following definitions.

\begin{defn}
For $S\subseteq \{1,\ldots,n\} \eqqcolon \undn$ let $f_S\colon n \to n$ be the endomorphism in $\cB(M,X)$ given by the constant path in $C_{n - \left| S\right|}(M,X)$ on the configuration $\{(a_i,x_0) \;|\; i\in \undn\smallsetminus S \}$. So this is the endomorphism which ``forgets'' the points $a_i$ for $i\in S$ and is the identity elsewhere.
\end{defn}

\begin{defn}\label{def:Tnk}
For $p\geq 0$ and $\lbrace S_1, \ldots, S_p \rbrace$ a partition of $S\subseteq\undn$ define
\[
T_n[S_1|\!\cdots\! |S_p] \;\coloneqq\; \image(Tf_{\undn\smallsetminus S}) \cap \bigcap_{i=1}^p \kernel(Tf_{S_i}).
\]
Note that the induced maps $Tf_S\colon T_n\to T_n$ are not in general $\bZ G_n$-module homomorphisms, so these are subgroups but not sub-$\bZ G_n$-modules.

We will write $S^\delta$ for the \emph{discrete} partition of $S$, and define
\[
T_n^k \;\coloneqq\; T_n[\{ n\!-\!k\!+\!1,\ldots,n \}^\delta].
\]
\end{defn}

\begin{rmk}\label{rObservations}
A few immediate observations are the following: Each $Tf_S\colon T_n\to T_n$ is idempotent. The composition of $Tf_{S_1}$ and $Tf_{S_2}$ is $Tf_{S_1 \cup S_2}$, so in particular the $Tf_S$ for $S\subseteq\undn$ all pairwise commute. By definition $T_n[\phantom{\cdot}] = \image(Tf_{\undn})$, and since $f_{\varnothing}=\mathrm{id}$ we also have $T_n[\undn] = \image(Tf_{\varnothing}) \cap \kernel(Tf_{\undn}) = \kernel(Tf_{\undn})$, so:
\begin{equation}\label{eDecomp0}
T_n = \image(Tf_{\undn}) \oplus \kernel(Tf_{\undn}) = T_n[\phantom{\cdot}] \oplus T_n[\undn].
\end{equation}
\end{rmk}

The following lemma is less immediate but can be proved by some diagram-chasing and drawing little cartoons like \eqref{eComposition} and \eqref{eIota}. We will give a proof in symbols.

\begin{lem}\label{lBijection}
For $k\leq m\leq n$, the map
\[
\iota_m^n \coloneqq \iota_{n-1} \circ \cdots \circ \iota_m \colon T_m \to T_n
\]
is split-injective and sends $T_m^k$ into $T_n^k$. Moreover, its restriction to a map $T_m^k \to T_n^k$ is a bijection. Hence any left-inverse for $\iota_m^n$ restricts to a bijection $T_n^k \to T_m^k$.
\end{lem}

\begin{proof}
As mentioned in \S\ref{ssDegree}, each $\iota_n$ has a natural left-inverse $\pi_n$ -- these compose to give a left-inverse $\pi_m^n$ for $\iota_m^n$. Just as for $\iota_n$ and $\pi_n$, by an abuse of notation we will denote the induced map $Tf_S\colon T_n\to T_n$ also by $f_S$.

We now show that $\iota_m(T_m^k)\subseteq T_{m+1}^k$, and hence by induction that $\iota_m^n(T_m^k)\subseteq T_n^k$. Suppose $x=\iota_m(y)$ for $y\in T_m^k$. Then by definition $y = f_{\{ 1,\ldots,m-k \}}(z)$ for some $z\in T_m$. Since $\pi_m\colon T_{m+1}\to T_m$ is split-surjective we have $z = \pi_{m}(w)$ for some $w\in T_{m+1}$. Hence
\begin{equation}\label{eProperty1}
x = \iota_m \circ f_{\{ 1,\ldots,m-k \}} \circ \pi_m (w) = f_{\{ 1,\ldots,m-k+1 \}} (w).
\end{equation}
For any $m-k+2 \leq i\leq m+1$ we have
\begin{equation}\label{eProperty2}
f_{\{i\}}(x) = f_{\{i\}} \circ \iota_m(y) = \iota_m \circ f_{\{i-1\}} (y) = \iota_m(0) = 0,
\end{equation}
since $y\in T_m^k$. The two properties \eqref{eProperty1} and \eqref{eProperty2} verify that $x\in T_{m+1}^k$.

Now we show that the restriction of $\iota_m$ to $T_m^k\to T_{m+1}^k$ is a bijection, and hence by induction that the restriction of $\iota_m^n$ to $T_m^k\to T_n^k$ is a bijection. Suppose $x\in T_{m+1}^k$, and define $z\coloneqq \pi_m(x)\in T_m$. Then
\[
\iota_m(z) = \iota_m \circ \pi_m (x) = f_{\{1\}} (x).
\]
But note that $x = f_{\{ 1,\ldots,m-k+1 \}}(y)$ for some $y\in T_{m+1}$, so
\begin{align*}
f_{\{1\}} (x) &= f_{\{1\}} \circ f_{\{ 1,\ldots,m-k+1 \}} (y) \\
\hspace*{2cm}&= f_{\{ 1,\ldots,m-k+1 \}} (y) \qquad \text{\small (by Remark \ref{rObservations} and since $k\leq m$)} \\
&= x.
\end{align*}
So it remains to prove that $z\in T_m^k$. Firstly,
\[
z = \pi_m \circ f_{\{ 1,\ldots,m-k+1 \}} (y) = f_{\{ 1,\ldots,m-k \}} \circ \pi_m (y).
\]
Secondly, for any $m-k+1 \leq i\leq m$, we have
\[
\iota_m \circ f_{\{i\}} (z) = f_{\{i+1\}} \circ \iota_m (z) = f_{\{i+1\}} (x) = 0,
\]
since $x\in T_{m+1}^k$. But $\iota_m$ is split-injective, so $f_{\{i\}}(z) = 0$. These two facts verify that $z\in T_m^k$.
\end{proof}

The following lemma will allow us to construct the required decomposition by induction:

\begin{lem}\label{lSES}
For all $\lbrace S_1, \ldots, S_p \rbrace$ partitioning $S\subseteq\undn$ with $p\geq 2$, there is a split short exact sequence
\[
0\to T_n[S_1|\!\cdots\! |S_p] \;\hookrightarrow\; T_n[S_1 \!\sqcup\! S_2|\!\cdots\! |S_p] \;\twoheadrightarrow\; T_n[S_1|S_3|\!\cdots\! |S_p] \;\oplus\; T_n[S_2|\!\cdots\! |S_p] \to 0.
\]
The first map is the inclusion, and a section of the second map is given by the inclusion of each of the two factors. So in other words we have a decomposition
\[
T_n[S_1 \!\sqcup\! S_2|\!\cdots\! |S_p] \;=\; T_n[S_1|\!\cdots\! |S_p] \;\oplus\; T_n[S_2|\!\cdots\! |S_p] \;\oplus\; T_n[S_1|S_3|\!\cdots\! |S_p].
\]
\end{lem}

\begin{proof}
One can check from the definitions that the following facts are true:
\begin{itemizeb}
\item[1.] $Tf_{S_2}$ restricts to a map $T_n[S_1\!\sqcup\! S_2|\!\cdots\! |S_p] \to T_n[S_1|S_3|\!\cdots\! |S_p]$,
\item[] and similarly $Tf_{S_1}$ restricts to a map $T_n[S_1\!\sqcup\! S_2|\!\cdots\! |S_p] \to T_n[S_2|\!\cdots\! |S_p]$.\vspace{0.2em}
\item[2.] $T_n[S_1|S_3|\!\cdots\! |S_p]$ and $T_n[S_2|\!\cdots\! |S_p]$ are contained in $T_n[S_1\!\sqcup\! S_2|\!\cdots\! |S_p]$.\vspace{0.2em}
\item[3.] For $\{i,j\} \subseteq \{1,2\}$ if $x\in T_n[S_i|S_3|\!\cdots\! |S_p]$, then $Tf_{S_j}(x)$ is $x$ when $i\neq j$ and $0$ when $i=j$.\end{itemizeb}
These facts imply that the map $(Tf_{S_2},Tf_{S_1})$ restricts to the required split surjection (with a section given by inclusion of each factor). The kernel of this is
\begin{align*}
& T_n[S_1\!\sqcup\! S_2|S_3|\!\cdots\! |S_p] \cap \kernel(Tf_{S_1}) \cap \kernel(Tf_{S_2}) \\
=\; & \image(Tf_{\undn\smallsetminus S}) \cap \bigcap_{i=3}^p \kernel(Tf_{S_i}) \cap \kernel(Tf_{S_1\sqcup S_2}) \cap \kernel(Tf_{S_1}) \cap \kernel(Tf_{S_2}) \\
=\; & T_n[S_1|\!\cdots\! |S_p],
\end{align*}
since $\kernel(Tf_{S_1}) \subseteq \kernel(Tf_{S_1\sqcup S_2})$.
\end{proof}

We can now use this to inductively prove a more general decomposition:

\begin{lem}\label{lDecompGeneral}
For any $\varnothing\neq S\subseteq\undn$ and $R\subseteq \undn\smallsetminus S$ there is a decomposition
\begin{equation}\label{eDecompGeneral}
T_n[S|R^\delta] = \bigoplus_{\varnothing\neq Q\subseteq S} T_n[(Q\!\sqcup\! R)^\delta].
\end{equation}
As before, $Q^\delta$ denotes the discrete partition of the set $Q$, so for example $T_n[\lbrace 1,2 \rbrace | \lbrace 3, 4, 5\rbrace^\delta]$ means $T_n[\lbrace 1,2\rbrace | \lbrace 3\rbrace | \lbrace 4\rbrace | \lbrace 5\rbrace]$. Note that this decomposition is an equality of subgroups, not just an abstract isomorphism of groups.
\end{lem}

\begin{proof}
The $|S|=1$ case is obvious, so we assume that $|S|\geq 2$ and assume the theorem for smaller values of $|S|$ by induction. Pick an element $s\in S$. Then by Lemma \ref{lSES},
\[
T_n[S|R^\delta] \;=\; T_n[S\!\smallsetminus\!\lbrace s\rbrace | (R\!\sqcup\!\lbrace s\rbrace)^\delta] \;\oplus\; T_n[S\!\smallsetminus\!\lbrace s\rbrace | R^\delta] \;\oplus\; T_n[\lbrace s\rbrace | R^\delta].
\]
Apply the inductive hypothesis to the right-hand side. The proposition then follows from the observation that for $\varnothing\neq Q\subseteq S$, exactly one of the following holds: (i) $s\in Q$ but $Q\neq\lbrace s\rbrace$; (ii) $s\notin Q$; (iii) $Q=\lbrace s\rbrace$.
\end{proof}

We can now use this to deduce the decomposition we want:

\begin{proof}[Proof of Proposition \ref{pDecomp}]
Combining \eqref{eDecompGeneral} (setting $R\coloneqq\varnothing$ and $S\coloneqq\undn$) with \eqref{eDecomp0} we obtain:
\begin{equation}\label{eDecompAdditive}
T_n = \bigoplus_{k=0}^n \; \bigoplus_{\substack{Q\subseteq \undn \\ \lvert Q \rvert = k}} T_n[Q^\delta].
\end{equation}
The action of $G_n$ on $T_n$ permutes the summands via the projection $G_n\to\Sigma_n$ and the obvious action of $\Sigma_n$ on subsets of $\undn$. So:
\begin{itemizeb}
\item[$\cdot$] $T_n^k = T_n[\{n\!-\!k\!+\!1,\ldots,n\}^\delta]$ is preserved by the action of $G_n^k \leq G_n$ on $T_n$.
\item[$\cdot$] The $G_n$-action on $T_n$ preserves the outer direct sum.
\item[$\cdot$] The inner direct sum is the induced module $\mathrm{Ind}_{G_n^k}^{G_n} T_n^k = \bZ G_n \otimes_{\bZ G_n^k} T_n^k$.
\end{itemizeb}
This establishes the decomposition of $\bZ G_n$-modules \eqref{eDecomp}. We proved in Lemma \ref{lBijection} above that $\iota_n \colon T_n \to T_{n+1}$ sends $T_n^k$ into $T_{n+1}^k$, and the naturality statement is clear.
\end{proof}

Having established this decomposition we can now define the \emph{height} of a twisted coefficient system:

\begin{defn}\label{dHeight}
The \emph{height} of a functor $T\colon\cB(M,X)\to\ab$ is the height at which the decomposition \eqref{eDecomp} is truncated. More precisely, we define $\mathrm{height}(T)$ by: $\mathrm{height}(T)\leq h$ if and only if $T_n^k = 0$ for all $k>h$ and all $n$. (So in particular $\mathrm{height}(T)=-1$ if and only if $T=0$.)
\end{defn}

\subsection{Height and degree.}
Despite their different definitions, these two notions are in fact equal:

\begin{lem}\label{lem:Equality}
For any functor $T\colon\cB(M,X)\to\ab$, $\mathrm{height}(T) = \mathrm{deg}(T)$.
\end{lem}

The important half of this equality is the inequality $\mathrm{height}(T) \leq \mathrm{deg}(T)$, since having an upper bound on the \emph{height} of a twisted coefficient system is what is needed to prove Theorem \ref{tMain}, whereas it is often easier to find an upper bound on the \emph{degree} in examples.

\begin{proof}
We will use induction on $d$ to prove the statement
\begin{equation}\label{eClaim}\tag{$\text{IH}_d$}
\degree(T)\leq d \;\Leftrightarrow\; \mathrm{height}(T)\leq d
\end{equation}
for all $d\geq -1$, using the decomposition \eqref{eDecompAdditive} above, which we restate as:
\begin{equation}\label{eDecompAdditiveRestated}
T_n \;=\; \bigoplus_{S\subseteq \undn} T_n[S^{\delta}].
\end{equation}
In this notation the height of $T$ is determined by saying that $\mathrm{height}(T)\leq d$ if and only if $T_n[S^\delta]=0$ for all $\lvert S \rvert > d$ and all $n$.

When $d=-1$ the definitions of height and degree coincide. This deals with the base case, so let $d\geq 0$ and assume that (\hyperref[eClaim]{$\text{IH}_{d-1}$}) holds. For all $n$ we have a split short exact sequence $0\to T_n \to T_{n+1} \to \Delta T_n \to 0$. Applying \eqref{eDecompAdditiveRestated}, this is
\[
0\to \bigoplus_{S\subseteq \undn} T_n[S^\delta] \longrightarrow \bigoplus_{R\subseteq \undnplusonescript} T_{n+1} [R^\delta] \longrightarrow \bigoplus_{Q\subseteq \undn} \Delta T_n [Q^\delta] \to 0.
\]
Analysing the maps carefully we see that
\begin{itemizeb}
\item[(a)] $T_n[S^\delta]$ is sent isomorphically onto $T_{n+1}[(S+1)^\delta]$ by the first map.
\item[(b)] $T_{n+1}[(Q\sqcup \{1\})^\delta]$ is sent isomorphically onto $\Delta T_n[(Q-1)^\delta]$ by the second map.
\end{itemizeb}

\noindent ($\Rightarrow$)
Suppose that $\degree(T)\leq d$. Then $\degree(\Delta T)\leq d-1$ by the definition of degree, and so by the inductive hypothesis (\hyperref[eClaim]{$\text{IH}_{d-1}$}), $\mathrm{height}(\Delta T)\leq d-1$. By fact (b) above this implies that
\begin{equation}\label{eStatement1}
T_{n+1}[R^\delta] = 0 \text{ whenever } \lvert R\rvert > d \text{ and } 1\in R.
\end{equation}
For any fixed $k$, the subgroups $\{ T_{n+1}[R^\delta] \;|\; \lvert R\rvert =k \}$ are all abstractly isomorphic via the action of $G_{n+1}$ on $T_{n+1}$. Also note that $d\geq 0$, so that $\lvert R\rvert > 0$, i.e.\ $R\neq\varnothing$. Hence:
\begin{equation}\label{eStatement2}
T_{n+1}[R^\delta] = 0 \text{ for \emph{all} } \lvert R\rvert > d.
\end{equation}
Therefore by (a), $T_n[S^\delta]=0$ for all $\lvert S\rvert > d$; in other words, $\mathrm{height}(T)\leq d$.

\noindent ($\Leftarrow$)
The other direction is simpler. Suppose that $\mathrm{height}(T)\leq d$; then we have the property \eqref{eStatement2} above. This imples that
\begin{equation}\label{eStatement3}
\Delta T_n [Q^\delta] \cong T_{n+1}[((Q+1)\sqcup \{1\})^\delta] = 0 \text{ for all } \lvert Q \rvert > d-1.
\end{equation}
Hence $\mathrm{height}(\Delta T)\leq d-1$, so we also have $\degree(\Delta T)\leq d-1$ by the inductive hypothesis (\hyperref[eClaim]{$\text{IH}_{d-1}$}). By definition this implies that $\degree(T)\leq d$, as required.
\end{proof}

\begin{rmk}\label{rmk:compare-height-and-degree}
The notion of \emph{height} in this paper is the same as the notion of degree in \cite{Betley2002Twistedhomologyof} (for twisted coefficient systems for symmetric groups) and \cite{Dwyer1980Twistedhomologicalstability} (for general linear groups), and goes back to Eilenberg and MacLane~\cite[\S 9]{EilenbergMac1954groupsHnII}. On the other hand, the notion of \emph{degree} in this paper is in the same spirit as the notion of degree in \cite{Ivanov1993homologystabilityTeichmuller}, \cite{CohenMadsen2009Surfacesinbackground} and \cite{Boldsen2012Improvedhomologicalstability} (for mapping class groups of surfaces) and \cite{Randal-WilliamsWahl2017Homologicalstabilityautomorphism} (for automorphism groups in a general categorical setting). See \cite{Palmer2017comparisontwistedcoefficient} for a detailed account of various notions of \emph{height} and \emph{degree} in the literature. Lemma \ref{lem:Equality} provides a link between these two different viewpoints (see also Remark 3.16 of \cite{Palmer2017comparisontwistedcoefficient}).
\end{rmk}

We finish this section with a few immediate facts about the degree of a twisted coefficient system.

\begin{lem}\label{lDegreeFacts}
For twisted coefficient systems $T,T^\prime \colon \cB(M,X)\to\ab$ and a fixed abelian group $A$,
\begin{itemizeb}
\item[\textup{(a)}] $\degree(T\oplus T^\prime) = \mathrm{max} \{ \degree(T), \degree(T^\prime) \}$,
\item[\textup{(b)}] $\degree(T\otimes A) \leq \degree(T)$,
\item[] and more generally, for $\degree(T)$ and $\degree(T^\prime)$ non-negative,
\item[\textup{(c)}] $\degree(T\otimes T^\prime) \leq \degree(T) + \degree(T^\prime)$,
\end{itemizeb}
where $\oplus$ and $\otimes$ are defined objectwise.
\end{lem}

\begin{proof}
Fact (a) follows by induction from the fact that $\Delta (T\oplus T^\prime) \cong \Delta T \oplus \Delta T^\prime$. Fact (b) follows from the fact that $\Delta (T\otimes A) \cong \Delta T \otimes A$, which is true because tensoring a \emph{split} short exact sequence with $A$ preserves split-exactness. Fact (c) is proved by induction with base case (b), and inductive step using the fact that
\[
\Delta (T\otimes T^\prime) \;\cong\; (T\otimes \Delta T^\prime) \oplus (\Delta T \otimes T^\prime) \oplus (\Delta T \otimes \Delta T^\prime).\qedhere
\]
\end{proof}


\section{Examples of twisted coefficient systems}\label{sExamples}

In this section we give some examples of twisted coefficient systems to which Theorem \ref{tMain} applies, and use them to deduce Corollaries \ref{cHomology} and \ref{cPartitions}. These examples are all pulled back from twisted coefficient systems for $\Sigma$ along the canonical functor $\cB(M,X) \to \Sigma$. After this, we also briefly discuss further examples of twisted coefficient systems, related to surface braid groups, which do not arise in this manner.

\subsection[Examples of twisted coefficient systems for Sigma and proofs of the corollaries]{Examples of twisted coefficient systems for $\Sigma$ and proofs of the corollaries.}

Recall (Definition \ref{dSigma}) that the category $\Sigma$ has non-negative integers as objects and partially-defined injections as morphisms. We will give some examples of functors $T\colon \Sigma\to\ab$, which are twisted coefficient systems for the special case $M=\bR^\infty$ and $X=*$ since $\cB(\bR^\infty)\cong\Sigma$.

Recall (see \S\ref{ssTCS-specialcase}) that there is a canonical functor $U\colon \cB(M,X)\to\Sigma$ for each $M$ and $X$ (\cf Remark 4.6 of \cite{Palmer2017comparisontwistedcoefficient}), so these examples also give twisted coefficient systems in general. Moreover, one may check (see \S\ref{sHeightDegree} for notation) that $\Delta(T\circ U) \cong \Delta T\circ U$, so by induction $\deg(T\circ U) = \deg(T)$ (\cf Lemma 4.2 of \cite{Palmer2017comparisontwistedcoefficient}), and also that $(T\circ U)_n^k \cong T_n^k$, so $\mathrm{height}(T\circ U) = \mathrm{height}(T)$.

\begin{eg}\label{egHomology}
Fix a path-connected based space $(Z,*)$, an integer $q\geq 0$ and a field $F$. The functor $\hat{T}_Z \colon \Sigma \to \mathsf{Top}$ is defined on objects by $n\mapsto Z^n$, and on morphisms as follows: given a partially-defined injection $j\colon \{1,\ldots,m\}\dashrightarrow \{1,\ldots,n\}$ in $\Sigma$, define $\hat{T}_Z (j) \colon Z^m \to Z^n$ to be the map
\[
(z_1, \ldots, z_m) \mapsto (z_{j^{-1}(1)}, \ldots, z_{j^{-1}(n)}),
\]
where $z_{\varnothing}$ is taken to mean the basepoint $*$. For example:
\begin{center}
\begin{tikzpicture}
[x=1mm,y=1mm]
\node (bl1) at (0,0) [fill,circle,inner sep=1pt] {};
\node (bl2) at (0,2) [fill,circle,inner sep=1pt] {};
\node (bl3) at (0,4) [fill,circle,inner sep=1pt] {};
\node (br1) at (10,0) [fill,circle,inner sep=1pt] {};
\node (br2) at (10,2) [fill,circle,inner sep=1pt] {};
\node (br3) at (10,4) [fill,circle,inner sep=1pt] {};
\node (br4) at (10,6) [fill,circle,inner sep=1pt] {};
\draw (bl1) .. controls (5,0) and (5,2) .. (br2);
\draw (bl2) .. controls (5,2) and (5,6) .. (br4);
\node at (13,2) [anchor=west] {$: \quad (z_1,z_2,z_3) \mapsto (*,z_1,*,z_2).$};
\end{tikzpicture}
\end{center}
The functor $T_{Z,q,F} \colon \Sigma\to\mathsf{Ab}$ is then the composite functor $H_q(-;F)\circ \hat{T}_Z$.
\end{eg}
\begin{lem}\label{lEgHomology}
The twisted coefficient system $T_{Z,q,F}$ has degree at most $\lfloor \tfrac{q}{h+1} \rfloor$, where for a path-connected space $Z$,
\[
h = \hconn_F(Z) \coloneqq \mathrm{max} \{ k\geq 0 \;|\; \widetilde{H}_i(Z;F)=0 \text{ for all } i\leq k \} \geq 0.
\]
\end{lem}

\begin{proof}
First note that the K{\"u}nneth theorem gives us natural split short exact sequences
\begin{equation}\label{eKunnethSES}
0 \to H_q(Z^n;F) \longrightarrow H_q(Z^{n+1};F) \longrightarrow \bigoplus_{i=1}^q H_{q-i}(Z^n;F) \otimes_F H_i(Z;F) \to 0,
\end{equation}
which together with the fact that $H_i(Z;F)=0$ for $1\leq i\leq h$ implies that
\begin{equation}\label{eDeltaT}
\Delta T_{Z,q,F} = \bigoplus_{i=h+1}^q T_{Z,q-i,F} \otimes_F H_i(Z;F).
\end{equation}
So, by Lemma \ref{lDegreeFacts} above, $\degree(T_{Z,q,F}) \leq 1 + \mathrm{max} \{ \degree(T_{Z,q-i,F}) \;|\; h+1\leq i\leq q \}$. Abbreviating $\degree(T_{Z,q,F})$ to $t_q$, we have the recurrence inequality
\begin{equation}\label{eRecurr1}
t_q \;\leq\; 1 + \mathrm{max} \{ t_0, \ldots, t_{q-h-1} \} .
\end{equation}
Note that $H_0(Z^n;F) \to H_0(Z^{n+1};F)$ is the identity map $F\to F$ for all $n$, so $\Delta T_{Z,0,F} = 0$, and hence $\degree(T_{Z,0,F})=0$. Also note that for $1\leq q\leq h$, $\hconn_F(Z)\geq q$ implies that $\hconn_F(Z^n)\geq q$ for all $n$ (by the K{\"u}nneth theorem), so $T_{Z,q,F}(n) = H_q(Z^n;F) = 0$, and hence $\degree(T_{Z,q,F})=-1\leq 0$. So we also have the initial conditions
\begin{equation}\label{eRecurr2}
t_0, t_1, \ldots, t_h \leq 0.
\end{equation}
It now remains to prove that the recurrence inequality \eqref{eRecurr1} and the initial conditions \eqref{eRecurr2} imply that $t_q \leq \lfloor \tfrac{q}{h+1} \rfloor$ for all $q\geq 0$. This will be done by induction on $q$. The base case is $0\leq q\leq h$ which is covered by the initial conditions \eqref{eRecurr2}. Assume that $q\geq h+1$. Then:
\begin{align*}
t_q &\leq 1 + \mathrm{max} \{ t_0, \ldots, t_{q-h-1} \} \\
&\leq 1+ \lfloor \tfrac{q-h-1}{h+1} \rfloor \\
&= \lfloor \tfrac{q}{h+1} \rfloor \qedhere
\end{align*}
\end{proof}

\begin{rmk}
See also \cite[Proposition 5.2]{Hanbury2009Homologicalstabilityof}, where it is proved (in the terminology of this paper) that the \emph{height} of $T_{Z,q,F}$ is at most $q$.
\end{rmk}

\begin{rmk}\label{rEgHomology}
If, in Lemma \ref{lEgHomology}, we replace $F$ by a general principal ideal domain $R$ (such as $\bZ$), the short exact sequence \eqref{eKunnethSES} becomes
\begin{align}\label{eKunnethSESb}
\begin{split}
0\to H_q(Z^n;R) \longrightarrow H_q(Z^{n+1};R) \longrightarrow &\bigoplus_{i=1}^q H_{q-i}(Z^n;R) \otimes_R H_i(Z;R) \\
\oplus &\bigoplus_{i=1}^q \mathrm{Tor}^R(H_{q-i}(Z^n;R),H_{i-1}(Z;R)) \to 0.
\end{split}
\end{align}
Here, as in \eqref{eKunnethSES}, we have used the splitting in the K{\"u}nneth short exact sequence to move some summands from the left-hand side to the right-hand side. However, this splitting is not always natural, and so \eqref{eKunnethSESb} is not natural for general principal ideal domains $R$. When $R=F$ is a field, the Tor terms vanish and the K{\"u}nneth short exact sequence is of the form $0\to A\to B\to 0\to 0$, so its splitting is certainly natural in this case.\footnote{This just comes from the fact that a natural transformation is invertible if it is objectwise invertible.} This is the reason why the short exact sequence \eqref{eKunnethSES} \emph{is} natural -- which was necessary to deduce the isomorphism of functors \eqref{eDeltaT}. More generally, the Tor terms vanish if $H_*(Z;R)$ is flat over $R$ in each degree, so the most general version of Example \ref{egHomology} works for a principal ideal domain $R$ and path-connected space $Z$ satisfying this condition. In particular, if $H_*(Z;\bZ)$ is torsion-free, this example works for homology with integral coefficients.
\end{rmk}

\begin{proof}[Proof of Corollary \ref{cHomology}]
The first statement follows directly from Theorem \ref{tMain} applied to Example~\ref{egHomology}, using Lemma \ref{lEgHomology} and Remark \ref{rEgHomology} to compute the degree of the twisted coefficient system in this case. The improved ranges follow from Remark \ref{rImprovedRanges}.
\end{proof}

\begin{notation}
Write $\bN = \bZ_{>0}$. For $k\in\bN$ and $\lambda = (\lambda_1,\ldots,\lambda_k)\in \bN^k$ define $\lvert\lambda\rvert \coloneqq \lambda_1 + \cdots + \lambda_k$. For $\ell\in\bN$, define $\lambda \vdash \ell$ to be the statement
\begin{center}
$\lambda\in\bN^k$ for some $k\in\bN$ and $\lvert\lambda\rvert = \ell$.
\end{center}
In words, $\lambda$ is an \emph{ordered partition} of $\ell$ of length $k$. For a set $S$ with $\lvert S\rvert \geq \ell$, an \emph{ordered decomposition} of $S$ of type $\lambda$ is a tuple $(S_1,\ldots,S_k)$ of pairwise disjoint subsets $S_i \subseteq S$ such that $\lvert S_i \rvert = \lambda_i$. Note that this decomposes $S$ into either $k$ or $k+1$ subsets, depending on whether $\lvert S\rvert = \ell$ or $\lvert S\rvert > \ell$. As a final piece of notation, define $\lambda[n]=\lambda$ for $n=\ell$ and
\[
\lambda[n] = (n-\ell,\lambda_1,\ldots,\lambda_k)
\]
for $n>\ell$, so that $\lambda[n] \vdash n$.
\end{notation}

\begin{eg}\label{egSubsets}
Let $\Se$ be the category of finite sets and partially-defined functions. Note that this is equivalent\footnote{Although not isomorphic, for essentially set-theoretic reasons.} to the category $\Setp$ of finite pointed sets. There is a free functor $\bZ (-)\colon \Se \to \mathsf{Ab}$ taking $S$ to $\bZ S$ and taking a partially-defined function $j\colon S \dashrightarrow R$ to the homomorphism
\begin{equation}\label{eEgSubsets}
\sum_{s\in S} n_s s \mapsto \sum_{s\in S} n_s j(s),
\end{equation}
where $j(s)$ means $0\in\bZ R$ if $j$ is undefined on $s$. So any functor $\Sigma \to \Se$ gives a twisted coefficient system for $\Sigma$ by composing with $\bZ (-)$.

We now define a functor $P_\lambda \colon \Sigma \to \Se$ associated to any $\lambda \vdash \ell$. On objects, it is defined by
\[
P_\lambda(n) = \begin{cases}
\{ \text{ordered decompositions of } \undn \text{ of type } \lambda \} & n\geq\ell \\
\varnothing & n<\ell.
\end{cases}
\]
Given a partially-defined injection $j\colon \{1, \ldots, m\} \dashrightarrow \{1, \ldots, n\}$, we define $P_\lambda(j)\colon P_\lambda(m) \dashrightarrow P_\lambda(n)$ as follows. First, if $m<\ell$ or $n<\ell$ then $P_\lambda(j)$ is the empty function. If $m,n\geq\ell$ and $(S_1,\ldots,S_k) \in P_\lambda(m)$, then $P_\lambda(j)$ is defined on $(S_1,\ldots,S_k)$ exactly when $j$ is defined on every element of $\bigcup_{i=1}^k S_i$, in which case its value is $(j(S_1),\ldots,j(S_k)) \in P_\lambda(n)$.

Note that, when $\lambda=1$, the functor $P_\lambda$ is simply the inclusion of $\Sigma$ as a subcategory of $\Se$. There is a natural action of $\Sigma_n$ on $P_\lambda(n)$, since $\Sigma_n$ is the automorphism group of $n$ in $\Sigma$, and an isomorphism
\[
\bZ P_\lambda(n) \;\cong\; \bZ \bigl[ \Sigma_n / \Sigma_{\lambda[n]} \bigr]
\]
of $\bZ[\Sigma_n]$-modules, where we write $\Sigma_\mu$ for the subgroup $\Sigma_{\mu_1} \times \cdots \times \Sigma_{\mu_k}$ of $\Sigma_{\lvert\mu\rvert}$. Note that the right-hand side is only defined for $n\geq\ell$. In particular, when $\lambda=(1,\ldots,1)$ with $\lvert\lambda\rvert = \ell$, we have $\bZ P_\lambda(n) \cong \bZ[\Sigma_n / \Sigma_{n-\ell}]$.

We have the following isomorphisms in $\mathsf{Ab}$ for $\lvert\lambda\rvert\geq 2$:
\begin{align*}
\Delta \bZ P_\lambda(n) &\cong \bZ \bigl\lbrace (S_1,\ldots,S_k)\in P_\lambda(n+1) \mid 1\in \textstyle{\bigcup}_{i=1}^k S_i \bigr\rbrace \\
&\cong \bZ \bigl( \textstyle{\bigsqcup}_{i=1}^k P_{\lambda-e_i}(n) \bigr) \\
&\cong \textstyle{\bigoplus}_{i=1}^k \bZ P_{\lambda-e_i}(n),
\end{align*}
where $\lambda-e_i$ is the ordered partition $(\lambda_1,\ldots,\lambda_i-1,\ldots,\lambda_k)$.\footnote{And where $(\lambda_1,\ldots,\lambda_k)$ means $(\lambda_1,\ldots,\lambda_{a-1},\lambda_{a+1},\ldots,\lambda_k)$ if $\lambda_a=0$.} The first and third isomorphisms are obviously natural isomorphisms of functors $\Sigma \to \mathsf{Ab}$, and one can also explicitly check that the second isomorphism is natural. Hence we have an isomorphism
\begin{equation}\label{eDeltaPlambda}
\Delta \bZ P_\lambda \;\cong\; \bigoplus_{i=1}^k \bZ P_{\lambda-e_i}
\end{equation}
for $\lvert\lambda\rvert\geq 2$. This allows us to prove:

\begin{lem}\label{lEgSubsets}
The twisted coefficient system $\bZ P_\lambda$ has degree $\lvert\lambda\rvert$.
\end{lem}

\begin{proof}
The proof is by induction on $\lvert\lambda\rvert$. First, if $\lvert\lambda\rvert = 1$ then $\Delta \bZ P_\lambda(n) \cong \bZ$ for all $n\geq 0$. Hence all morphisms in $\Sigma$ are sent by $\Delta \bZ P_\lambda$ to endomorphisms of $\bZ$ in $\mathsf{Ab}$. But all morphisms in $\Sigma$ have one-sided inverses, so their images in $\mathsf{Ab}$ are endomorphisms of $\bZ$ admitting one-sided inverses, and hence automorphisms. Thus $\Delta\Delta \bZ P_\lambda = 0$, and so $\bZ P_\lambda$ has degree $1$ by definition.

Now assume that $\lvert\lambda\rvert \geq 2$. By \eqref{eDeltaPlambda}, Lemma \ref{lDegreeFacts} and the inductive hypothesis, we have:
\[
\deg(\Delta \bZ P_\lambda) = \deg \Bigl( \bigoplus_{i=1}^k \bZ P_{\lambda-e_i} \Bigr) = \max_{i=1,\ldots,k} (\deg(\bZ P_{\lambda-e_i})) = \lvert\lambda\rvert -1,
\]
so $\deg(\bZ P_\lambda) = \lvert\lambda\rvert$ by the definition of degree.
\end{proof}
\end{eg}

\begin{rmk}\label{rEgSubsets}
Given an arbitrary ring $R$, there is also a functor $R(-)\colon \Se \to \mathsf{Ab}$ taking a set $S$ to the free $R$-module generated by $S$ (viewed as an abelian group) and with morphisms defined by the same formula \eqref{eEgSubsets} as for $\bZ(-)$. Thus we have twisted coefficient systems $RP_\lambda \colon \Sigma \to \mathsf{Ab}$ associated to any ring $R$ and ordered partition $\lambda$. Just as in the case $R=\bZ$, we have isomorphisms $RP_\lambda(n) \cong R[\Sigma_n/\Sigma_{\lambda[n]}]$ of $R[\Sigma_n]$-modules for all $n$, and the twisted coefficient system $RP_\lambda$ has degree $\lvert \lambda \rvert$. To see this, we can adapt the proof of Lemma \ref{lEgSubsets} directly, as long as we are slightly more careful about the base case. It is not in general true that the monoid $\mathrm{End}_{\mathsf{Ab}}(R)$ has the property that any one-sided inverse is a two-sided inverse (consider $R=\prod^\infty \bZ$ for example), so the base case does not come for free. However, one can explicitly compute the maps $\Delta RP_\lambda(m) \to \Delta RP_\lambda(n)$ induced by any $j\colon \{1, \ldots, m\} \dashrightarrow \{1, \ldots, n\}$ in $\Sigma$, and see that they are just the identity on $R$.
\end{rmk}

\begin{proof}[Proof of Corollary \ref{cPartitions}]
The first statement follows directly from Theorem \ref{tMain} applied to Example~\ref{egSubsets}, using Lemma \ref{lEgSubsets} and Remark \ref{rEgSubsets} to compute the degree of the twisted coefficient system in this case. The improved ranges follow from Remark \ref{rImprovedRanges}.
\end{proof}

\subsection{Examples of twisted coefficient systems for surface braid groups.}\label{ss:examples-braid-groups}

\paragraph*{Examples derived from LKB representations.}

The \emph{Lawrence-Krammer-Bigelow representations} are a family of representations of the braid groups first introduced by Lawrence~\cite{Lawrence1990HomologicalrepresentationsHecke}. They have since been studied by many people, including Bigelow~\cite{Bigelow2001Braidgroupslinear} and Krammer~\cite{Krammer2002Braidgroupslinear}, who proved (independently) that a certain one of these representations is faithful, thereby proving the linearity of the braid groups. These representations depend on a choice of positive integer $m$, and come in several flavours, defined as the (ordinary, reduced, Borel-Moore) homology of a certain covering space of the configuration space $C_m(\bC - \{ 1,2,\ldots,n \})$, with the $n$th braid group acting via compactly-supported diffeomorphisms of the punctured plane.

These representations, as well as further variants allowing more general punctured surfaces (giving rise to representations of surface braid groups), will be studied in detail in future work. Here we just remark that certain special cases are known to assemble into finite-degree twisted coefficient systems on the category $\Ubeta$ (\cf Remark \ref{rRelatedResults}), including the unreduced Burau representations \cite[Example 4.15]{Randal-WilliamsWahl2017Homologicalstabilityautomorphism}, the reduced Burau representations~\cite[Corollary 2.36]{Soulie2017LongMoodyconstruction} (both corresponding to $m=1$) and the \emph{Lawrence-Krammer representations}~\cite[Proposition 2.40]{Soulie2017LongMoodyconstruction} (corresponding to $m=2$).

\paragraph*{Further examples for braid groups.}

Another example of a finite-degree twisted coefficient system on $\Ubeta$ is constructed in \cite[Proposition 2.29]{Soulie2017LongMoodyconstruction} from the \emph{Tong-Yang-Ma representations}~\cite{TongYangMa1996newclassrepresentations}. Moreover, the main result of \cite{Soulie2017LongMoodyconstruction} is a functorial version of the \emph{Long-Moody construction} \cite{Long1994Constructingrepresentationsbraid}, which produces a new twisted coefficient system on $\Ubeta$ from an old one, increasing the degree by exactly one in the process. Iterating this construction therefore gives many more examples of finite-degree twisted coefficient systems on $\Ubeta$.


\section{A twisted Serre spectral sequence}\label{sTSSS}

To prove Theorem \ref{tMain} we will need a generalisation of the basic Serre spectral sequence, allowing the base space to be equipped with a local coefficient system. It is a special case of (the homology version of) an \emph{equivariant} generalisation of the Serre spectral sequence constructed by Moerdijk and Svensson in \cite{MoerdijkSvensson1993equivariantSerrespectral}. This section gives a brief description of their spectral sequence and deduces the particular case that we will need. We note that the required spectral sequence may also be deduced as a special case of \cite[\S 20.4]{MaySigurdsson2006Parametrizedhomotopytheory}, but we will not do this explicitly here.

\subsection{Recollections about homology with local coefficients.}\label{ss:homology-with-local-coefficients}

We start by recalling a basepoint-independent description of (co)homology with local coefficients (in the non-equivariant setting). Let $R$ be a commutative ring with unit. We first describe homology and cohomology of categories, which we view as functors
\begin{equation}\label{eq:cohomology-of-categories}
H_* \colon \catr \longrightarrow \grrmod \qquad\text{and}\qquad H^* \colon \catrr \longrightarrow \grrmod
\end{equation}
respectively, where the target is the category of graded $R$-modules and the sources are as follows. An object of $\catr$ is a category $C$ and a functor $F \colon C \to R\text{-mod}$ and a morphism from $(C,F)$ to $(D,G)$ is a functor $\phi \colon C \to D$ and a natural transformation $F \Rightarrow G \circ \phi$. An object of $\catrr$ is a category $C$ and a functor $F \colon C^{\mathrm{op}} \to R\text{-mod}$ and a morphism from $(C,F)$ to $(D,G)$ is a functor $\phi \colon D \to C$ and a natural transformation $F \circ \phi^{\mathrm{op}} \Rightarrow G$. These are defined using the derived functors Ext and Tor for representations of categories \cite{Mitchell1972Ringswithseveralobjects,BauesWirsching1985Cohomologysmallcategories}:\footnote{In \cite[\S 12]{Mitchell1972Ringswithseveralobjects}, the (co)homology of a category $C$ is defined more generally, allowing coefficients in any functor $C^{\mathrm{op}} \times C \to R\text{-mod}$, and in \cite{BauesWirsching1985Cohomologysmallcategories} this is further generalised to functors $FC \to R\text{-mod}$, where $FC$ is a certain category of ``factorisations in $C$''.}
\[
H_*(C,F) \coloneqq \mathrm{Tor}_*^C(R,F) \qquad\text{and}\qquad H^*(C,F) \coloneqq \mathrm{Ext}^*_C(R,F),
\]
where $R$ denotes the constant functor $C \to R\text{-mod}$ respectively $C^{\mathrm{op}} \to R\text{-mod}$ sending every object to the free $R$-module on one generator and sending every morphism to the identity. They may also be computed using the classical Ext and Tor functors for modules over rings, since there is an embedding $\mathsf{Fun}(C,R\text{-mod}) \hookrightarrow RC\text{-mod}$ for any category $C$ and ring $R$, where $RC$ is the \emph{category ring} of $C$ (this was defined in \cite[\S 7]{Mitchell1972Ringswithseveralobjects}, see also \cite[\S 5.1]{Palmer2017comparisontwistedcoefficient}):
\[
H_*(C,F) = \mathrm{Tor}_*^{RC}(R,F) \qquad\text{and}\qquad H^*(C,F) = \mathrm{Ext}^*_{RC^{\mathrm{op}}}(R,F),
\]
where $R$ is considered as an $RC$-module (respectively $RC^{\mathrm{op}}$-module) with the trivial action of the category ring.

\begin{defn}
For a space $Y$ let $\Delta(Y)$ be the category whose objects are all singular simplices in $Y$, and whose morphisms are simplicial operations (generated by face and degeneracy maps). For example, $\Delta(*)$ is the usual simplex category. Denote the fundamental groupoid of $Y$ by $\pi(Y)$, and the standard $n$-simplex by $\Delta^n = \{ (x_0,\ldots,x_n) \in \bR^{n+1} \mid x_i \geq 0,\; x_0 + \cdots + x_n = 1 \}$. There is a canonical functor $v_Y\colon \Delta(Y)\to \pi(Y)$ which takes a singular simplex $\Delta^n\to Y$ to the image of its barycentre $b_n = \bigl(\tfrac{1}{n+1},\ldots,\tfrac{1}{n+1}\bigr)$. A morphism $\Delta^k\xrightarrow{\alpha}\Delta^n\to Y$ is taken to the image of the straight-line path in $\Delta^n$ from $\alpha(b_k)$ to $b_n$. (One may alternatively define a functor using the last vertex $e_n = (0,\ldots,0,1)$ in place of the barycentre $b_n$ --- this is naturally isomorphic to $v_Y$.)

A \emph{coefficient system} for homology (resp.\ cohomology) is a covariant (resp.\ contravariant) functor $\Delta(Y) \to R\text{-mod}$. It is a \emph{local coefficient system} if it factors up to natural isomorphism through $v_Y$. Homology and cohomology of spaces with local coefficients are then functors
\begin{equation}\label{eq:cohomology-of-spaces}
H_* \colon \topr \longrightarrow \grrmod \qquad\text{and}\qquad H^* \colon \toprr \longrightarrow \grrmod,
\end{equation}
where an object of $\topr$ is a space $Y$ and a functor $F \colon \pi(Y) \to R\text{-mod}$ and a morphism from $(Y,F)$ to $(Z,G)$ is a continuous map $f \colon Y \to Z$ and a natural transformation $F \Rightarrow G \circ \Delta(f)$. An object of $\toprr$ is a space $Y$ and a functor $F \colon \pi(Y)^{\mathrm{op}} \to R\text{-mod}$ and a morphism from $(Y,F)$ to $(Z,G)$ is a continuous map $f \colon Z \to Y$ and a natural transformation $F \circ \Delta(f)^{\mathrm{op}} \Rightarrow G$. The definition is:
\[
H_*(Y,F) \coloneqq H_*(\Delta(Y),F \circ v_Y) \qquad\text{and}\qquad H^*(Y,F) \coloneqq H^*(\Delta(Y),F \circ v_Y).
\]
The \emph{homotopy-invariance} of (co)homology may then be expressed by the statement that the functors \eqref{eq:cohomology-of-spaces} are \emph{continuous} functors between topologically-enriched categories, where the topology on morphism-sets in $\topr$ and $\toprr$ is defined in the obvious way using the compact-open topology of mapping spaces, and the topology on morphism-sets in $\grrmod$ is discrete.
\end{defn}

\begin{rmk}\label{rmk:Davis-Kirk}
If we restrict to topological spaces that are locally path-connected and semi-locally simply-connected, then the categories $\topr$ and $\toprr$ are equivalent to the categories $\cL$ and $\cL^*$ mentioned in \S 5.4 of \cite{DavisKirk2001Lecturenotesin} (see Theorems 5.11 and 5.12 respectively). In particular $\topr$ may be thought of as follows: an object is a space (satisfying the above conditions) equipped with a bundle of $R$-modules over it; a morphism is a continuous map of spaces covered by a morphism of bundles that restricts to an $R$-linear isomorphism on each fibre.
\end{rmk}

\subsection{The spectral sequence.}\label{ss:TSSS}

In \cite{MoerdijkSvensson1993equivariantSerrespectral} the above is generalised to the equivariant setting: they define certain categories and a functor $v_Y\colon \Delta_G(Y) \to \pi_G(Y)$ for a $G$-space $Y$, and equivariant twisted cohomology $H^*_G(Y;F)$ for any coefficient system $F \colon \Delta_G(Y)^{\mathrm{op}}\to\mathsf{Ab}$. Again a coefficient system is \emph{local} if it factors up to natural isomorphism through $v_Y$. Cohomology with respect to local coefficient systems is $G$-homotopy invariant \cite[Theorem 2.3]{MoerdijkSvensson1993equivariantSerrespectral}. Their main theorem is the existence of a twisted equivariant Serre spectral sequence:

\begin{thm}[{\cite[Theorem 3.2]{MoerdijkSvensson1993equivariantSerrespectral}}]\label{Ttssscohomology}
For any $G$-fibration $f\colon Y\to X$ \textup{(}i.e.\ $Y^H\to X^H$ is a fibration for all $H\leq G$\textup{)} and any local coefficient system $F$ on $Y$, there is a local coefficient system $H^q_G(f;F)$ on $X$ for each $q\geq 0$ and a spectral sequence
\begin{equation}\label{Etssscohomology}
E_2^{p,q} = H^p_G\bigl(X;H^q_G(f;F)\bigr) \;\Rightarrow\; H^*_G(Y;F)
\end{equation}
with the usual cohomological grading.
\end{thm}

\begin{rmk}
We describe the local coefficient system $H^q(f;F)$ in the non-equivariant case. As a functor $\Delta(X)^{\mathrm{op}}\to\mathsf{Ab}$ it acts as follows. An object (i.e.\ singular simplex $\Delta^k\xrightarrow{\sigma}X$) is taken to the cohomology $H^q(\sigma^*(Y);F)$, where $\sigma^*(Y)$ is the pullback of $\sigma$ and $f$, and we denote any pullback of the coefficients $F$ also by $F$. A morphism $\Delta^l\xrightarrow{\alpha}\Delta^k\xrightarrow{\sigma}X$ induces a map of pullbacks $(\sigma\circ\alpha)^*(Y)\to \sigma^*(Y)$ and hence a map on cohomology. This coefficient system is a \emph{local} coefficient system since it factors up to natural isomorphism through $v_X$ by the following functor $\pi(X)^{\mathrm{op}}\to\mathsf{Ab}$. A point $x\in X$ is taken to $H^q(f^{-1}(x);F)$. Given a homotopy class $[I\xrightarrow{p}X]$ of paths from $x$ to $y$, there are induced maps of pullbacks $f^{-1}(x)\hookrightarrow p^*(Y)\hookleftarrow f^{-1}(y)$. These induce maps on cohomology, and since they are \emph{isomorphisms}\footnote{The inclusion $\{0\}\hookrightarrow [0,1]$ is an acyclic cofibration, so its pullback along the fibration $f$ is again an acyclic cofibration, in particular a weak equivalence.} the first one can be inverted to get a composite map $H^q(f^{-1}(x);F)\to H^q(f^{-1}(y);F)$. One can check that this map is independent of the choice of representing path $p$.
\end{rmk}

In \cite{MoerdijkSvensson1993equivariantSerrespectral} the authors point out that there is an analogous version of the spectral sequence \eqref{Etssscohomology} for homology. We will only need the non-equivariant (but twisted) version, which is:\footnote{This was also stated (referencing \cite{MoerdijkSvensson1993equivariantSerrespectral}) as Theorem 4.1 of \cite{Hanbury2009openclosedcobordism}.}

\begin{thm}\label{Ttssshomology}
For any fibration $f\colon Y\to X$ and any local coefficient system $F$ on $Y$, there is a local coefficient system $H_q(f;F)$ on $X$ for each $q\geq 0$ and a spectral sequence
\begin{equation}\label{Etssshomology}
E^2_{p,q} = H_p\bigl(X;H_q(f;F)\bigr) \;\Rightarrow\; H_*(Y;F)
\end{equation}
with the usual homological grading.
\end{thm}

The description of the local coefficient systems $H_q(f;F)$ is the same as above, replacing cohomology with homology. An important observation is that if the local coefficient system $F$ on $Y$ is pulled back from the base $X$, the local coefficient systems $H_q(f;F)$ are built out of the \emph{untwisted} homology of each fibre.

We now return to the viewpoint --- in the setting of based, path-connected spaces --- of local coefficient systems as an action of the fundamental group on an abelian group. In the special case where the local coefficient system on $Y$ is a pullback of one on $X$ the above can be rephrased as:

\begin{coro}\label{Ctssshomology}
For any fibration $f\colon Y\to X$ with fibre $F$ over the basepoint $x_0\in X$, and any $\pi_1(X)$-module $M$, there is a spectral sequence
\begin{equation}\label{Etssshomology2}
E^2_{p,q} = H_p\bigl(X;H_q(F;M)\bigr) \;\Rightarrow\; H_*(Y;M)
\end{equation}
with the usual homological grading. Here the action of $\pi_1(Y)$ on $M$ is pulled back from that of $\pi_1(X)$ via $f_*$ and the action of $\pi_1(F)$ on $M$ is trivial. The action of $\pi_1(X)$ on $H_q(F;M)$ is induced by its diagonal action on the chain complex $S_*(X)\otimes_{\bZ} M$.
\end{coro}

This is natural for maps of fibrations in the obvious way:

\begin{prop}\label{Ptsssnaturality}
Suppose we have a map of fibrations \textup{(}the vertical maps are fibrations, and the square commutes on the nose\textup{):}
\begin{center}
\begin{tikzpicture}
[x=0.8mm,y=1mm,>=stealth']
\node (tl) at (-10,10) {$Y^{\phantom{\prime}}$};
\node (tr) at (10,10) {$Y^\prime$};
\node (bl) at (-10,0) {$X$};
\node (br) at (10,0) {$X^\prime$};
\draw[->] (tl) to (tr);
\draw[->] (bl) to (br);
\draw[->] (tl) to (bl);
\draw[->] (tr) to (br);
\end{tikzpicture}
\end{center}
and a $\pi_1(X^\prime)$-module $M$. Denote the fibres over the basepoints by $F$ and $F^\prime$ respectively. Then there is a map of spectral sequences \eqref{Etssshomology2} where\textup{:}
\begin{itemizeb}
\item[$\circ$] The map $F\to F^\prime$ induces a map of untwisted homology $H_q(F;M)\to H_q(F^\prime;M)$, which is equivariant w.r.t.\ the homomorphism $\pi_1(X)\to \pi_1(X^\prime)$, so it induces a map of twisted homology $H_p(X;H_q(F;M))\to H_p(X^\prime;H_q(F^\prime;M))$. This is the map on the $E^2$ pages.
\item[$\circ$] The action of $\pi_1(Y)$ on $M$ is the pullback of the action of $\pi_1(Y^\prime)$ on $M$, so the map $Y\to Y^\prime$ induces a map of twisted homology $H_*(Y;M)\to H_*(Y^\prime;M)$. This is the map in the limit.
\end{itemizeb}
\end{prop}


\section{Proof of twisted homological stability}\label{sProof}

We now use the twisted Serre spectral sequence of the previous section to prove Theorem \ref{tMain}. We first record another fact we will use:

\begin{lem}[Shapiro for covering spaces]\label{lShapiro}
Suppose we have a based, path-connected space $X$ which is locally path-connected and semi-locally simply-connected, a subgroup $H$ of $\pi_1(X)$ and an $H$-module $A$. Let $\hat{X}$ be the \textup{(}based\textup{)} covering space corresponding to $H$. Then
\begin{equation}\label{eShapiro}
H_*(\hat{X}; A) \cong H_*(X; \bZ \pi_1(X)\otimes_{\bZ H} A).
\end{equation}
Moreover, given a map of the above data, namely a \textup{(}based\textup{)} map $f\colon X\to X^\prime$ such that $f_*(H)\subseteq H^\prime$ \textup{(}so that there is a unique based lift $\hat{f} \colon \hat{X} \to \hat{X}^\prime$\textup{)} and a map $\phi \colon A\to A^\prime$ which is equivariant w.r.t.\ $f_*$, the identification \eqref{eShapiro} is natural in the sense that
\begin{equation}\label{eShapiroNaturality}
\centering
\begin{split}
\begin{tikzpicture}
[x=1mm,y=1mm,>=stealth']
\node (tl) at (0,0) {$H_*(X; \bZ \pi_1(X) \otimes_{\bZ H} A)$};
\node (tr) at (50,0) {$H_*(X^\prime; \bZ \pi_1(X^\prime) \otimes_{\bZ H^\prime} A^\prime)$};
\node (bl) at (0,10) {$H_*(\hat{X}; A)$};
\node (br) at (50,10) {$H_*(\hat{X}^\prime; A^\prime)$};
\draw[->] (tl) to (tr);
\draw[->] (bl) to (br);
\node at (0,5) {\rotatebox[origin=c]{270}{$\cong$}};
\node at (50,5) {\rotatebox[origin=c]{270}{$\cong$}};
\end{tikzpicture}
\end{split}
\end{equation}
commutes.
\end{lem}

\begin{proof}
Denote the singular chain complex functor by $S_*(\phantom{-})$ and the universal cover of $X$ by $\widetilde{X}$. Then we have an isomorphism of chain complexes
\[
S_*(\widetilde{X}) \otimes_{\bZ H} A \longrightarrow S_*(\widetilde{X}) \otimes_{\bZ \pi_1(X)} \bZ \pi_1(X) \otimes_{\bZ H} A
\]
given by $\sigma\otimes a \mapsto \sigma\otimes [c_x] \otimes a$, where $c_x$ is the constant loop at the basepoint $x$ of $X$. Taking homology gives the identification \eqref{eShapiro}. Let $\widetilde{f}$ denote the unique (based) lift of $f$ to $\widetilde{X}\to \widetilde{X}^\prime$. The diagram \eqref{eShapiroNaturality} is induced by
\begin{center}
\begin{tikzpicture}
[x=1.2mm,y=1.2mm,>=stealth']
\node (bl) at (0,0) {$S_*(\widetilde{X}) \otimes_{\bZ \pi_1(X)} \bZ \pi_1(X) \otimes_{\bZ H} A$};
\node (br) at (50,0) {$S_*(\widetilde{X}^\prime) \otimes_{\bZ \pi_1(X^\prime)} \bZ \pi_1(X^\prime) \otimes_{\bZ H^\prime} A^\prime$};
\node (tl) at (0,10) {$S_*(\widetilde{X}) \otimes_{\bZ H} A$};
\node (tr) at (50,10) {$S_*(\widetilde{X}^\prime) \otimes_{\bZ H^\prime} A^\prime$};
\draw[->] (tl) to (tr);
\draw[->] (bl) to (br);
\draw[->] (tl) to node[left,font=\small]{$\cong$} (bl);
\draw[->] (tr) to node[right,font=\small]{$\cong$} (br);
\end{tikzpicture}
\end{center}
and one can check that both routes around the square send $\sigma \otimes a$ to $\widetilde{f}_\sharp (\sigma) \otimes [c_{x^\prime}] \otimes \phi(a)$.
\end{proof}

This will be applied to the following covering spaces of configuration spaces:

\begin{defn}
The configuration space $C_{(k,n-k)}(M,X)$ of $k$ red and $n-k$ green points in $M$ with labels in $X$ is defined to be
\[
(\mathrm{Emb}(n,M) \times X^n)/(\Sigma_{n-k} \times \Sigma_k)
\]
(\cf Remark \ref{rColoured}), and we give it the basepoint $\{(a_1,x_0),\ldots,(a_n,x_0)\}$ with the points $a_1,\ldots,a_{n-k}$ coloured green and the points $a_{n-k+1},\ldots,a_n$ coloured red. There is also a stabilisation map $s_n^k\colon C_{(k,n-k)}(M,X) \to C_{(k,n-k+1)}(M,X)$, which is defined exactly as in \S\ref{ssTCS-stabmap}, and adds a new green point to the configuration.
\end{defn}

\begin{defn}
Let $f\colon C_{(k,n-k)}(M,X)\to C_k(M,X)$ be the map which forgets the green points. We will also need the following two maps for technical reasons: Define $p\colon C_k(M,X)\to C_k(M,X)$ to be the self-homotopy-equivalence induced by the self-embedding $e|_{M}\colon M\emb M$ (see \S\ref{ssTCS-setup}). Choose a self-diffeomorphism of $M$ which is isotopic to the identity and which takes $a_i$ to $a_{i+n-k+1}$ for $i=1,\ldots,k$. Denote by $\phi$ the self-homeomorphism $C_k(M,X)\to C_k(M,X)$ induced by this.
\end{defn}

The forgetful maps $f$ are locally trivial fibre bundles, so we have a map of fibrations:

\begin{equation}\label{eMapOfFibrations}
\centering
\begin{split}
\begin{tikzpicture}
[x=2mm,y=1.5mm,>=stealth']
\node (tl) at (0,10) {$C_{(k,n-k)}(M,X)$};
\node (tr) at (20,10) {$C_{(k,n-k+1)}(M,X)$};
\node (bl) at (0,0) {$C_k(M,X)$};
\node (br) at (20,0) {$C_k(M,X)$};
\draw[->] (tl) to node[above,font=\small]{$s_n^k$} (tr);
\draw[->] (bl) to node[below,font=\small]{$\phi^{-1}\circ p$} (br);
\draw[->] (tl) to node[left,font=\small]{$f$} (bl);
\draw[->] (tr) to node[right,font=\small]{$\phi^{-1}\circ f$} (br);
\end{tikzpicture}
\end{split}
\end{equation}

The $p$ is there to ensure that it commutes on the nose, and the $\phi^{-1}$ is there to deal with basepoints: on the bottom-left we have to give $C_k(M,X)$ the basepoint $\{(a_{n-k+1},x_0),\ldots,(a_n,x_0)\}$, but on the bottom-right we can give it its usual basepoint of $\{(a_1,x_0),\ldots,(a_k,x_0)\}$.

The map $s_n^k$ restricted to the fibres over the basepoints is a map
\[
C_{n-k}(M\smallsetminus\{a_{n-k+1},\ldots,a_n\},X) \to C_{n-k+1}(M\smallsetminus\{a_{n-k+2},\ldots,a_{n+1}\},X),
\]
but this can be identified, up to homeomorphism, with the stabilisation map $s_{n-k}\colon C_{n-k}(M_k,X) \to C_{n-k+1}(M_k,X)$, where $M_k$ is $M$ with a subset of $M\smallsetminus U$ of size $k$ removed (see \S\ref{ssTCS-setup} for notation).

Finally, before beginning the proof proper, we mention how a certain local coefficient system pulls back along the maps in \eqref{eMapOfFibrations}. The covering space $C_{(k,n-k)}(M,X)\to C_n(M,X)$ corresponds to the subgroup $G_n^k \leq G_n = \pi_1 C_n(M,X)$. Recall from Proposition \ref{pDecomp} that $T_n^k$ is a $\bZ G_n^k$-module (it is a sub-$\bZ G_n^k$-module of $T_n$), so it is a local coefficient system for $C_{(k,n-k)}(M,X)$.

\begin{lem}\label{lPullbackLCS}
The local coefficient system $T_k^k$ on the right-hand base space pulls back to the local coefficient systems $T_n^k$ and $T_{n+1}^k$ on the total spaces of \eqref{eMapOfFibrations}.
\end{lem}

\begin{proof}
By Lemma \ref{lBijection}, the left-inverse $\pi_k^n$ of $\iota_k^n\colon T_k\to T_n$ restricts to a bijection $T_n^k\to T_k^k$. So this is an isomorphism of abelian groups, and it is enough to check that it is equivariant w.r.t.\ the map on $\pi_1$ induced by the composite $\phi^{-1}\circ p\circ f$ in \eqref{eMapOfFibrations}. This is true because both $e|_{M}\colon M\emb M$ (which induces $p$) and the diffeomorphism which induces $\phi$ are isotopic to the identity. Exactly the same argument works for the right-hand side.
\end{proof}

\begin{proof}[Proof of Theorem \ref{tMain} \textup{(}except the split-injectivity claim\textup{)}]
We need to show that the map
\begin{equation}\label{Emap1}
H_*(C_n(M,X); T_n) \longrightarrow H_*(C_{n+1}(M,X); T_{n+1})
\end{equation}
induced by $s_n$ and $\iota_n$ is an isomorphism in the range $*\leq\frac{n-d}{2}$. By the decomposition \eqref{eDecomp} of Proposition \ref{pDecomp}, and the fact that $T$ has degree $d$, this is the same as the map
\begin{equation}\label{Emap2}
\bigoplus_{k=0}^d H_*(C_n(M,X); \bZ G_n \otimes_{\bZ G_n^k} T_n^k) \longrightarrow \bigoplus_{k=0}^d H_*(C_{n+1}(M,X); \bZ G_{n+1} \otimes_{\bZ G_{n+1}^k} T_{n+1}^k)
\end{equation}
induced by $s_n$, $\iota_n$ and $(s_n)_*$. By Shapiro's Lemma for covering spaces (Lemma \ref{lShapiro}) this is isomorphic to the map
\begin{equation}\label{Emap3}
\bigoplus_{k=0}^d H_*(C_{(k,n-k)}(M,X); T_n^k) \longrightarrow \bigoplus_{k=0}^d H_*(C_{(k,n-k+1)}(M,X); T_{n+1}^k)
\end{equation}
induced by $s_n^k$ and $\iota_n$. The map of fibrations \eqref{eMapOfFibrations} gives the following map of twisted Serre spectral sequences (Corollary \ref{Ctssshomology}, Proposition \ref{Ptsssnaturality} and Lemma \ref{lPullbackLCS}):
\begin{equation}\label{Emtsss}
\centering
\begin{split}
\begin{tikzpicture}
[x=1mm,y=1.2mm,>=stealth']
\node (t1) at (0,10) {$E^2_{p,q} = H_p(C_k(M,X); H_q(C_{n-k}(M_k,X); T_k^k))$};
\node (t2) at (70,10) {$H_*(C_{(k,n-k)}(M,X); T_n^k)$};
\node (b1) at (0,0) {$E^2_{p,q} = H_p(C_k(M,X); H_q(C_{n-k+1}(M_k,X); T_k^k))$};
\node (b2) at (70,0) {$H_*(C_{(k,n-k+1)}(M,X); T_{n+1}^k).$};
\draw[->] (t1) to (b1);
\draw[->] (t2) to (b2);
\node at ($ (t1.east)!0.5!(t2.west) $) {$\Rightarrow$};
\node at ($ (b1.east)!0.5!(b2.west) $) {$\Rightarrow$};
\end{tikzpicture}
\end{split}
\end{equation}
The map in the limit is the $k$th summand of \eqref{Emap3}, and the map on $E^2$ pages is induced by the stabilisation map $s_{n-k}$ on the fibres and the homotopy-equivalence $\phi^{-1}\circ p$ on the base. Note that $T_k^k$ is a \emph{constant} coefficient system once it has been pulled back to the fibres $C_{n-k}(M_k,X)$ and $C_{n-k+1}(M_k,X)$, since it was originally pulled back from the base.

Hence, by \emph{untwisted} homological stability for configuration spaces (Theorem \ref{tUntwistedStability}) and the universal coefficient theorem, the map on $E^2$ pages is an isomorphism for $q\leq\frac{n-k}{2}$ (and all $p\geq 0$). By the Zeeman comparison theorem\footnote{The required implication is contained in the proof of Theorem 1 of \cite{Zeeman1957proofofcomparison}, although stronger hypotheses are stated there. An explicit statement of the comparison theorem which applies to our case is Theorem 1.2 of \cite{Ivanov1993homologystabilityTeichmuller}. It is also written in Remarque 2.10 of \cite{CollinetDjamentGriffin2013Stabilitehomologiquepour}.} it is therefore an isomorphism in the limit for $*\leq\frac{n-k}{2}$. So in the range $*\leq\frac{n-d}{2}$ each summand in \eqref{Emap3} is an isomorphism, so \eqref{Emap1} is an isomorphism.
\end{proof}

\begin{rmk}\label{rRationalRange}
When $M$ is at least $3$-dimensional, the stabilisation map $C_n(M,X)\to C_{n+1}(M,X)$ is an isomorphism on homology with coefficients in $\bZ[\frac12]$ in the larger range $*\leq n$, by \cite{KupersMiller2015Improvedhomologicalstability}. Since $\bZ[\frac12]$ is a PID, this implies, via the universal coefficient theorem, the same for homology with coefficients in any $\bZ[\frac12]$-module. Suppose that $T\colon \cB(M,X) \to \bZ[\frac12]\text{-}\mathsf{mod} \leq \mathsf{Ab}$ is a twisted coefficient system of $\bZ[\frac12]$-modules (meaning that it takes values in the full subcategory $\bZ[\frac12]\text{-}\mathsf{mod}$ of $\mathsf{Ab}$) of degree $d$. Then the constant coefficients $T_k^k$ appearing in \eqref{Emtsss} above are all $\bZ[\frac12]$-modules, and the same proof tells us that the map
\begin{equation}\label{eImprovedRange}
H_*(C_n(M,X);T_n) \to H_*(C_{n+1}(M,X);T_{n+1})
\end{equation}
is an isomorphism in the larger range $*\leq n-d$ (rather than just $*\leq\frac{n-d}{2}$). When $M$ is a surface, there is a similar improvement to the range for rational coefficients. In this case the stabilisation map is an isomorphism on homology with rational coefficients in the range $*\leq n$ in the non-orientable case and in the range $*<n$ in the orientable case, by \cite[Corollary 3]{Church2012Homologicalstabilityconfiguration} and \cite[Theorem 1.3]{Knudsen2014Bettinumbersand}.\footnote{The maps used in these two references to induce isomorphisms between configuration spaces are not the stabilisation maps. However, we may reduce to the case where the manifolds are of finite type, so that the rational homology of the configuration spaces is a finite-dimensional vector space in each degree. Moreover, the stabilisation maps are always split-injective in all degrees (see Theorem \ref{tUntwistedStability}). So the fact that $H_*(C_n(M,X);\bQ)$ and $H_*(C_{n+1}(M,X);\bQ)$ are abstractly isomorphic in a range implies that the stabilisation map is an isomorphism in this range.} Thus if $T\colon \cB(M,X) \to \mathsf{Vect}_{\bQ} \leq \mathsf{Ab}$ is a \emph{rational} twisted coefficient system of degree $d$, the map \eqref{eImprovedRange} is an isomorphism in either the range $*\leq n-d$ (for non-orientable surfaces) or the range $*<n-d$ (for orientable surfaces).
\end{rmk}


\section{Split-injectivity}\label{sInjectivity}

To prove the split-injectivity part of Theorem \ref{tMain} we will use the following lemma which was used implicitly by Nakaoka in \cite{Nakaoka1960Decompositiontheoremhomology} and later written down explicitly by Dold in \cite{Dold1962DecompositiontheoremsSn}:

\begin{lem}[{\cite[Lemma 2]{Dold1962DecompositiontheoremsSn}}]\label{lDold}
Given a sequence $0\to A_1 \xrightarrow{\phi_1} A_2 \xrightarrow{\phi_2} \cdots$ of abelian groups and homomorphisms, the following is sufficient to imply that each of the maps $\phi_i$ is split-injective\textup{:} There exist maps $\tau_{k,n}\colon A_n \to A_k$ for $1\leq k\leq n$ with $\tau_{n,n} = \mathrm{id}$ such that
\begin{equation}\label{eDold}
\image(\tau_{k,n} - \tau_{k,n+1}\circ \phi_n) \leq \image(\phi_{k-1}).
\end{equation}
\end{lem}

Let $U_n(M,X)$ be the universal cover of $C_n(M,X)$. One can think of its elements as $n$-strand ``open-ended braids'' in $M\times [0,1]$ ($n$ pairwise disjoint paths in $M\times [0,1]$ which are the identity in the second coordinate and start at $\{(a_1,0),\ldots,(a_n,0)\}$, up to endpoint-preserving homotopy) with each strand labelled by the based path space $PX$. Let $\widetilde{s}_n \colon U_n(M,X)\to U_{n+1}(M,X)$ be the lift of the stabilisation map which applies $e|_{M}\times \mathrm{id}_{[0,1]}$ to the braid and adds a vertical strand at $a_1$ labelled by the constant path $c_{x_0}$.

As before, denote $\pi_1 C_n(M,X)$ by $G_n$, and denote the singular chain complex of a space by $S_*(\phantom{-})$. Let $T\colon \cB(M,X)\to\ab$ be any twisted coefficient system (we do not assume finite-degree in this section). Then the map
\begin{equation}\label{Estabrepeated}
(s_n;\iota_n)_* \colon H_*(C_n(M,X);T_n) \longrightarrow H_*(C_{n+1}(M,X);T_{n+1}).
\end{equation}
is induced by the map of chain complexes
\begin{equation*}
(\widetilde{s}_n)_{\sharp} \otimes \iota_n \colon S_*(U_n(M,X)) \otimes_{\bZ G_n} T_n \longrightarrow S_*(U_{n+1}(M,X)) \otimes_{\bZ G_{n+1}} T_{n+1}.
\end{equation*}

\begin{proof}[Proof of Theorem \ref{tMain} \textup{(}split-injectivity claim\textup{)}]
We want to prove that \eqref{Estabrepeated} is split-injective for all $*$ and $n$. By Dold's Lemma \ref{lDold}, it is sufficient to construct chain maps
\[
t_{k,n}\colon S_*(U_n(M,X)) \otimes_{\bZ G_n} T_n \longrightarrow S_*(U_k(M,X)) \otimes_{\bZ G_k} T_k
\]
for $1\leq k\leq n$ such that $t_{n,n} = \mathrm{id}$ and
\begin{equation}\label{Edold2}
t_{k,n} \simeq t_{k,n+1} \circ ((\widetilde{s}_n)_{\sharp} \otimes \iota_n) - ((\widetilde{s}_{k-1})_{\sharp} \otimes \iota_{k-1}) \circ t_{k-1,n}.
\end{equation}

Let $S\subseteq\{1,\ldots,n\}$. There is a unique partially-defined injection $\{1, \ldots, n\} \dashrightarrow \{1, \ldots, \scard\}$ which is order-preserving and is defined precisely on $S$. This is a morphism $n\to\scard$ in the category $\Sigma$. Let $\pi_{S,n}$ be its lift along $\cB(M,X)\to\Sigma$ to a morphism $n \to \scard$ of $\cB(M,X)$ given by travelling along the paths $p_i$ (see \S\ref{ssTCS-setup}) and keeping the labels constant. By our standard abuse of notation we will denote its image under $T$ also by $\pi_{S,n}\colon T_n \to T_{\scard}$.

We also define a map $p_{S,n}\colon U_n(M,X) \to U_{\scard}(M,X)$ as follows. Given an open-ended braid in $U_n(M,X)$, forget the strands which start at $(a_i,0)$ for $i\in \{1,\ldots,n\}\smallsetminus S$, and then concatenate this with the reverse of $\pi_{S,n} \colon n \to \scard$ to get an open-ended braid in $U_{\scard}(M,X)$.

Directly from these definitions one can check (where the notation $(S-1)$ means $\{ s-1 \,|\, s\in S\}$):
\begin{itemizeb}
\item[(a)] If $1\notin S$ then $\pi_{S,n+1}\circ \iota_n = \pi_{(S-1),n}$ and $p_{S,n+1}\circ\widetilde{s}_n \simeq p_{(S-1),n}$.
\item[(b)] If $1\in S$ then $\pi_{S,n+1}\circ \iota_n = \iota_{\lvert S\rvert -1}\circ \pi_{(S\smallsetminus\{1\}-1),n}$ and $p_{S,n+1}\circ\widetilde{s}_n = \widetilde{s}_{\lvert S\rvert -1}\circ p_{(S\smallsetminus\{1\}-1),n}$.
\end{itemizeb}
We now define $t_{k,n}$ to be the following chain map:
\begin{equation*}
\sigma \otimes x \;\mapsto\; \sum_{S\subseteq \{1,\ldots,n\},\, \scard = k} (p_{S,n})_{\sharp} (\sigma) \otimes \pi_{S,n} (x) .
\end{equation*}
Clearly $t_{n,n}=\mathrm{id}$, so we just need to check the identity \eqref{Edold2}. The right-hand side of this is:
\begin{equation}\label{Ehuge1}
\begin{split}
\sigma \otimes x \quad\mapsto\quad &\sum_{S\subseteq \{1,\ldots,n+1\},\, \scard =k} \bigl( (p_{S,n+1})_{\sharp} \circ (\widetilde{s}_n)_{\sharp} (\sigma) \bigr) \otimes \bigl( \pi_{S,n+1} \circ \iota_n (x) \bigr) \\
-&\sum_{R\subseteq \{1,\ldots,n\},\, \rcard =k-1} \bigl( (\widetilde{s}_{k-1})_{\sharp} \circ (p_{R,n})_{\sharp} (\sigma) \bigr) \otimes \bigl( \iota_{k-1} \circ \pi_{R,n} (x) \bigr) .
\end{split}
\end{equation}
Using (a) and (b) above, we see that the top line of this decomposition is chain-homotopic to:
\begin{equation}\label{Ehuge2}
\begin{split}
\sigma \otimes x \;\;\mapsto\quad &\sum_{S\subseteq \{1,\ldots,n+1\},\, \scard =k,\, 1\in S} \bigl( (\widetilde{s}_{k-1})_{\sharp} \circ (p_{(S\smallsetminus\{1\}-1),n})_{\sharp} (\sigma) \bigr) \otimes \bigl( \iota_{k-1} \circ \pi_{(S\smallsetminus\{1\}-1),n} (x) \bigr) \\
+&\sum_{S\subseteq \{1,\ldots,n+1\},\, \scard =k,\, 1\notin S} (p_{(S-1),n})_{\sharp} (\sigma) \otimes \pi_{(S-1),n} (x) .
\end{split}
\end{equation}
The first line of \eqref{Ehuge2} cancels with the second line of \eqref{Ehuge1}, leaving just the second line of \eqref{Ehuge2}, which is precisely $t_{k,n}$, as required.
\end{proof}


\phantomsection
\addcontentsline{toc}{section}{References}
\renewcommand{\bibfont}{\normalfont\small}
\setlength{\bibitemsep}{0pt}
\printbibliography

\end{document}